 \documentclass[draft]{article}

\usepackage{amsmath,amsfonts,amsthm,amssymb,amscd,cancel,color}
\usepackage{enumitem}
\usepackage{verbatim}
\usepackage[dvipdfmx]{graphicx}
\usepackage{ulem,color}

\theoremstyle{plain}
\newtheorem{theorem}{Theorem}[section]
\newtheorem{lemma}[theorem]{Lemma}
\newtheorem{proposition}[theorem]{Proposition}

\theoremstyle{definition}

\theoremstyle{remark}
\newtheorem{remark}[theorem]{Remark}
\newtheorem{notation}[theorem]{Notation}

\newtheorem{assumption}[theorem]{Assumption}

\newcommand{\R}{{\mathbb R}}

\newcommand{\N}{{\mathbb N}}

\def\im{{\rm i}}

\newcommand{\C}{\mathbb{C}}

\def\({\left(}
\def\){\right)}
\def\<{\left\langle}
\def\>{\right\rangle}
\newcommand{\sech}{{\mathrm{sech}}}

\numberwithin{equation}{section}

\begin{document}

\title{Small energy stabilization for 1D Nonlinear Klein Gordon Equations}

\author{Scipio Cuccagna, Masaya Maeda, Stefano Scrobogna}
\maketitle

\begin{abstract}
We give a  partial extension to dimension 1 of the result proved by   Bambusi and Cuccagna \cite{BC11AJM} on the absence of small energy real valued periodic solutions for  the NLKG  in dimension 3. We combine the framework in
 Kowalczyk and Martel \cite{KM22} with the notion of "refined profile".
\end{abstract}

\section{Introduction}

Let $m>0$ and $V\in \mathcal{S}(\R,\R)$ (Schwartz function) with set of eigenvalues
\begin{align}\label{eq:spec}
\sigma_{\mathrm{d}}(L_1)=\{\lambda_j^2\ |\ j=1,\cdots,N\} \text{ with } 0<\lambda_1<\cdots<\lambda_N<m,\text{ where } L_1=-\partial_x^2+V+m^2.
\end{align}
 We assume there exist   $C>0$ and $a_1>0$ such that
\begin{align}\label{eq:decay} |V ^{(l)}(x)|\le C  e^{-a_1 |x|} \text{ for all $0\le l\le N+1$.}
 \end{align} Let $f\in C^\infty(\R,\R)$ s.t.\ $f(0)=f'(0)=0$.
We consider the nonlinear Klein-Gordon (NLKG) equation
\begin{align}\label{NLKG}
\dot{\mathbf{u}}=\mathbf{J} \(\mathbf{L}_1 \mathbf{u}+\mathbf{f}[\mathbf{u}]\),\ \mathbf{u}={}^t(u_1\ u_2):\R\times\R\to \R^2,
\end{align}
where
\begin{align*}
\mathbf{J}=\begin{pmatrix}
0 & 1 \\ -1 & 0
\end{pmatrix},\quad
\mathbf{L}_1=\begin{pmatrix}
L_1 & 0 \\ 0 & 1
\end{pmatrix},\quad
\mathbf{f}[\mathbf{u}]=\begin{pmatrix}
f(u_1)\\ 0
\end{pmatrix}.
\end{align*}
Denoting by  $\phi_j$  a   real valued  eigenfunction with $L^2(\R )$ norm equal to 1 of $L_1$ associated to $\lambda_j^2$, setting
\begin{align}\label{eq:matr_eig}
\boldsymbol{\Phi}_j:=\begin{pmatrix}
\phi_j\\ \im \lambda_j \phi_j
\end{pmatrix}   \text{    for $j=1,\cdots,N$},
\end{align}
  we have
\begin{align}\label{eq:matr_eig1}
\mathbf{J}\mathbf{L}_1 \boldsymbol{\Phi}_j=\im \lambda_j\boldsymbol{\Phi}_j    \text{   and } \mathbf{J}\mathbf{L}_1 \overline{\boldsymbol{\Phi}}_j=-\im \lambda_j\overline{\boldsymbol{\Phi}}_j.
\end{align}
In fact the $ \boldsymbol{\Phi} _j$  and  their complex conjugates $ \overline{\boldsymbol{\Phi}} _j$  generate all the eigenfunctions of the linearization $\mathbf{J}\mathbf{L}_1 $  of our NLKG \eqref{NLKG}.

\noindent Our NLKG \eqref{NLKG} is a  Hamiltonian system for the symplectic form
\begin{align} &
\Omega(\mathbf{u},\mathbf{v}):=\<\mathbf{J}^{-1}\mathbf{u},\mathbf{v}\>, \text{  where } \<\mathbf{u},\mathbf{v}\> := Re\(\mathbf{u},\overline{\mathbf{v}}\)   \text{  and } \label{eq:inner1} \\&  \label{eq:inner0}  \(
 \mathbf{u},  \mathbf{ v}\)  :=\int _{\R} {^t\mathbf{u}}(x) \mathbf{ v}(x) dx,
\end{align}
and the Hamiltonian  or  energy function is given by
\begin{align}\label{eq:energy}
E(\mathbf{u})=\frac{1}{2}\<\mathbf{L}_1\mathbf{u},\mathbf{u}\>+\int_{\R}F(u_1)\,dx, \text{ where $F(u)=\int_0^sf(\tau)\,d\tau$.}
\end{align}
The local well-posedness of \eqref{NLKG} is well known.
From the conservation of the energy, we have that for sufficiently small $\delta>0$, if $\|\mathbf{u}_0\|_{\boldsymbol{\mathcal{H}}^1}\leq \delta$, then $\|\mathbf{u}\|_{L^\infty (\R ,\boldsymbol{\mathcal{H}}^1)}\lesssim \delta$ and in particular we obtain  the global well-posedness for small data, where \begin{align}\label{eq:ennrom}
 \|\mathbf{u}\|_{\boldsymbol{\mathcal{H}^1} }^2=\|u_1\|_{H^1}^2+\|u_2\|_{L^2}^2.
\end{align}

\noindent Given a constant $a>0$ we consider the space defined by the norm
\begin{align}\label{eq:enwei}
 \| \mathbf{u}\|_{ \boldsymbol{\mathcal{H}}^{1 }_{-a}}:= \| \sech \( ax \)   \mathbf{u}\|_{ \boldsymbol{\mathcal{H}}^{1 } }.
\end{align}
  We   denote by $\boldsymbol{\phi} [  \mathbf{z} ]  $ the \textit{refined profile},  introduced below in Sect.\ \ref{sec:refprof}, where
\begin{align}\label{eq:discr_mores}
		 \mathbf{z}= (z_1,...,z_{ {N}}  ),
	\end{align}
encodes the discrete modes and where $\boldsymbol{\phi} [  \mathbf{z} ]= \sum z_j\boldsymbol{\Phi}  _j +c.c. +O(\| \mathbf{z} \| )  $,
 where by $g +c.c.$, we mean $g+\bar{g}$ and $\|\mathbf{z}\|^2:=\sum_{j=1}^{N}|z_j|^2$.

 The main result in this paper is the following theorem.

\begin{theorem}\label{thm:main}
Under Assumptions \ref{ass:generic}, \ref{ass:FGR} and \ref{ass:repulsive} given below, for any $a>0$ and $ \epsilon >0$
  there exists $\delta_0>0$ such that  if $\|\mathbf{u}_0\|_{\boldsymbol{\mathcal{H}}^1}=:\delta<\delta_0$, then
we have a global representation
 \begin{align}
 \mathbf{u}(t) =     {\boldsymbol{\phi}}[ \mathbf{z}(t)]   + \boldsymbol{\eta }
   (t) \text{ for appropriate $\mathbf{z}\in C ^1(\R , \C ^{N})$ and $ \boldsymbol{\eta }\in C^0 (\R , \boldsymbol{\mathcal{H}}^{1})$,}
 \label{eq:main1}
\end{align}
and, for  $I=\R$,
\begin{align}
 \int _{I }  \|     \boldsymbol{\eta }  (t) \| ^2 _{\boldsymbol{\mathcal{H}}^{1} _{-a}(\R )} dt   \le \epsilon,     \label{eq:main2}
\end{align}
   and
\begin{align}&
\lim_{t\to  \infty} \mathbf{z} (t) =0     \text{  .  }   \label{eq:main3}
\end{align}
\end{theorem}

The result of this paper is a partial extension to dimension 1 of the result, on local decay to zero for small  real valued solutions of an NLKG with a trapping potential and, in particular, on the absence of small energy real valued periodic solutions, proved   for dimension 3 by   Bambusi and Cuccagna \cite{BC11AJM}.
 The latter was an
extension, to cases with quite general spectral configurations, of a result proved by Soffer and Weinstein \cite{SW3}
under rather restrictive spectral hypotheses. There is a substantial literature on the asymptotic stability of patterns for wave like equations,
partially reviewed for the case of the Nonlinear Schr\"odinger Equation (NLS) in \cite{CM21DCDS}. In particular, in a series of papers referenced in
\cite{CM21DCDS}, we have expanded the result of \cite{BC11AJM} to various contexts  where dispersion can be proved using Strichartz estimates.
The crux of these papers consisted in proving a form of radiation induced damping on the discrete modes of the system (the so called Nonlinear Fermi Golden Rule, or FGR), due  to the spilling of the energy in the discrete modes in the radiation
component of the solutions, where dispersion occurs because of linear dispersion. Recently, thanks to the notion introduced
in \cite{CM20}, of \textit{Refined Profile}, we have
been able to simplify significantly the proofs, see also \cite{CM2111.02681,CM21}, eliminating the normal forms arguments required to find a coordinates system where
 the FGR can be seen. In fact, an ansatz involving the Refined Profile yields authomatically
 a framework adequate to prove the FGR, as we will see later.

 Lately,  in the literature
there has been  considerable attention on low dimensional problems, especially in 1D, where, due to the
relative strength of the nonlinearities,  the Strichartz estimates are not sufficient to prove dispersion. Various papers
like for example  \cite{DM20}--\cite{Martel2110.01492}, \cite{Snelson} and \cite{Sterbenz}
have recently dealt with asymptotic stability problems in the context of long range nonlinearities.  In
\cite{CM2109.08108,CM2203.13468} use is made  of the theory of Virial Inequalities developed by Kowalczyk et al.
\cite{KM22}--\cite{KMMvdB21AnnPDE}.  In this paper    we will follow closely   Kowalczyk and Martel \cite{KM22}.
So,  as  in  \cite{KM22}--\cite{KMMvdB21AnnPDE}, we will need two distinct sets of Virial Inequalities.
We follow the Kowalczyk and Martel \cite{KM22} idea of proving the FGR utilizing the initial sets of coordinates, contrary to what
 is done in \cite{CM2109.08108,CM2203.13468}. In particular, in the proof of the FGR we use a functional derived from Kowalczyk and Martel \cite{KM22}, instead of
 the localized energy $E( \boldsymbol{\phi} [  \mathbf{z} ])$.
 The proof simplifies, avoiding the use of the smoothing estimates,
 which played a significant role in  \cite{CM2109.08108,CM2203.13468}. We highlight that our result works under a somewhat restrictive hypothesis on the potential
 $V$, specifically that the potential $V_D$ obtained after eliminating all the eigenvalues of $L_1$ with a sequence of Darboux transformations, must be a repulsive
 potential, in the sense of Assumption \ref{ass:repulsive}.

\subsection{Assumptions and refined profile}\label{sec:refprof}

\begin{notation}
We write $a\lesssim b$ to mean that there exists a constant $C>0$ s.t.\ $a\leq Cb$.
The positive number $C$ omitted is called the implicit constant.
\end{notation}

We set $\boldsymbol{\lambda}=(\lambda_1,\cdots,\lambda_{N},-\lambda_1,\cdots,-\lambda_{N})\in \R^{2N}$ and
\begin{align*}
\mathbf{R}&:=\{\mathbf{m}=(\mathbf{m}_+,\mathbf{m}_-)\in (\N\cup \{0\})^{2N}\ |\  |\mathbf{m} \cdot \boldsymbol{\lambda}|> m\},\\
\mathbf{R}_{\mathrm{min}}&:=\{\mathbf{m}\in \mathbf{R}\ |\ \not\exists \mathbf{n}\in \mathbf{R}\ \text{s.t.}\ \mathbf{n}\prec\mathbf{m}\},\\
\mathbf{I}&:=\{\mathbf{m}\in (\N\cup \{0\})^{2N}\ |\ \exists \mathbf{n}\in\mathbf{R}_{\mathrm{min}}\ \text{s.t.}\ \mathbf{n}\prec \mathbf{m}\},
\end{align*}
where
\begin{align*}
&\mathbf{n}=(\mathbf{n}_+,\mathbf{n}_-)\prec \mathbf{m}=(\mathbf{m}_+,\mathbf{m}_-)\\& \Longleftrightarrow\ \forall j=1,\cdots,N,\ n_{+,j}+n_{-,j}\leq m_{+,j}+m_{-,j} \ \text{and}\ \|\mathbf{n}\|<\|\mathbf{m}\|,  \\&  \text{where } \|\mathbf{m}\|:=\sum_{j=1}^{N}\sum_{\pm} m_{\pm,j}.
\end{align*}
We also set $\mathbf{e}^j=(\delta_{1j},\cdots,\delta_{Nj},0,\cdots,0)$ where $\delta_{jk}$ is the Kronecker's delta, $\overline{\mathbf{m}}=\overline{(\mathbf{m}_+,\mathbf{m}_-)}:=(\mathbf{m}_-,\mathbf{m}_+)$ and
\begin{align*}
\mathbf{NR}&:=(\N\cup \{0\})^{2N} \setminus(\mathbf{R}_{\mathrm{min}}\cup \mathbf{I}),\\
\boldsymbol{\Lambda}_j&:=\{\mathbf{m}\in \mathbf{NR}\ |\ \mathbf{m} \cdot \boldsymbol{\lambda}=\lambda_{j}\},\\
\overline{\boldsymbol{\Lambda}_j}&:=\{\overline{\mathbf{m}}\ |\ \mathbf{m}\in \boldsymbol{\Lambda}_j\}\\
 \boldsymbol{\Lambda}_0&:=\{\mathbf{m} \in \mathbf{NR} \backslash \{ \mathbf{0} \}\ |\ \boldsymbol{\lambda}\cdot \mathbf{m}=0\}.
\end{align*}
We assume the following, which is true for generic $L_1$.
\begin{assumption}\label{ass:generic}For  $M  $   the largest number in $\N$ such that $(M -1)\lambda _1< m$, then for a multi--index $ \mathbf{m} \in \N_0^{2N}$ we assume
\begin{align}&  \|  \mathbf{m}  \|  \le M   \Longrightarrow      \(   \mathbf{m} \cdot \boldsymbol{\lambda} \) ^2 \neq m ^2  . \label{eq:generic12}
\end{align}
We also assume that    for $\mathbf{m}=( \mathbf{m}_+,  \mathbf{m}_-) \in \N_0^{2N}$ then
\begin{align}&  \| \mathbf{m} \|  \le 2M \text{  and    }  \mathbf{m}\cdot \boldsymbol{\lambda} = 0 \Longrightarrow       \mathbf{m}_+ =  \mathbf{m}_- . \label{eq:generic11}
\end{align}
\end{assumption}
\begin{lemma}\label{lem:combinat1}   The following facts hold.
 \begin{enumerate}
   \item  If $\| \mathbf{m} \|> M$, with $M$ the constant in Assumption \ref{ass:generic},   then  $\mathbf{m}\in \mathbf{I}$.
   \item  $\mathbf{R}_{\mathrm{min}}$ and  $\mathbf{NR}$ are finite sets.
   \item If $\mathbf{m}\in \mathbf{NR}$, then $|\boldsymbol{\lambda}\cdot \mathbf{m}|<m$ and if $\mathbf{m}\in \mathbf{R}_{\mathrm{min}}$, then $\mathbf{m}_+=0$ or $\mathbf{m}_-=0$.
   \item  If $\mathbf{m}\in \boldsymbol{\Lambda}_{j}$ then  there is   a $\mathbf{n} \in \boldsymbol{\Lambda} _{0}$ with $\mathbf{m}= \mathbf{e}^j+\mathbf{n}$.
\end{enumerate}
 \end{lemma} \proof The proof is taken from \cite{CM2203.13468}.If
  $\| \mathbf{m} \|> M$,  we can write $\mathbf{m} =\boldsymbol{\alpha  } +\boldsymbol{ \beta}  $  with
   $\|\boldsymbol{\alpha  } \| = M $.  If $   \boldsymbol{\alpha  }=(\boldsymbol{\alpha  }_+,\boldsymbol{\alpha  }_-)$ and if we set $   \mathbf{n}=(\mathbf{n}_+,\mathbf{n}_-)$  with $\mathbf{n}_+= \boldsymbol{\alpha  }_++\boldsymbol{\alpha  }_-$ and
   $\mathbf{n}_-= 0$, then $\mathbf{n} \cdot \boldsymbol{\lambda} \ge M \lambda _1 >  m   $. This implies that $ \mathbf{{n  } }\in    \mathbf{R} $ and that there exists $\boldsymbol{\mathfrak{ a} }\in \mathbf{R}_{\mathrm{min} }$  with $\boldsymbol{\mathfrak{ a} } \preceq  \mathbf{n}$. From  $\| \boldsymbol{\beta  }\| \ge 1 $ it follows that   $\boldsymbol{\mathfrak{ a} } \prec \mathbf{m}$ and so $\mathbf{m}\in \mathbf{I}$.

   Obviously, from the 1st claim it follows that if  $\mathbf{m}\in  \mathbf{R}_{\mathrm{min}}\cup \mathbf{NR}$ then $\| \mathbf{m} \|\le  M$. Next we observe that $\mathbf{m}\in \mathbf{NR}$ implies $\| \mathbf{m} \|\le  M$ and $|\boldsymbol{\lambda}\cdot \mathbf{m}| \le m$ and, by Assumption \ref{ass:generic}, $|\boldsymbol{\lambda}\cdot \mathbf{m}| < m$.
    If $\mathbf{m}\in \mathbf{R}_{\mathrm{min}}$ with, say, $\mathbf{m} \cdot \boldsymbol{\lambda} >m$, then obviously   we have
 $\mathbf{m} _+ \cdot \boldsymbol{\lambda} >m$ and it is elementary that $\mathbf{m} =(\mathbf{m} _+,0)$.
 Finally,  from the first claim we know that if $\mathbf{m}\in \boldsymbol{\Lambda}_{j}$ then
   $\| \mathbf{m} \| \le  M$. From $  \mathbf{m}   \cdot \boldsymbol{\lambda  } -\lambda _j =0$ it follows
   from \eqref{eq:generic11}  that  we have the last claim.\qed



For $\mathbf{z}=(z_1,\cdots,z_{N})\in \C^{N}$ and $\mathbf{m}\in(\N\cup \{0\})^{2N}$, we set
\begin{align*}
\mathbf{z}^{\mathbf{m}}=\prod_{j=1}^{N}z_j^{m_{+,j}}\bar{z}_j^{m_{-,j}}.
\end{align*}
Notice that we have $\overline{\mathbf{z}^{\mathbf{m}}}=\mathbf{z}^{\overline{\mathbf{m}}}$.

Notice that $ \sum_{j=1}^{N}\(z_j \boldsymbol{\Phi}_j+c.c.\), $ satisfies \eqref{NLKG} up to $O(\|\mathbf{z}\|^2)$ error if $\dot{z}_j=\im \lambda_{j}z_j$. The \textit{refined profile} is a generalization of  this kind of  approximate solution  of  \eqref{NLKG} .

  We set   $\|\cdot\|_{\Sigma^s}:=\| \cdot \|_{H^s_{a_2}}:=\| e^{a_2\<x\>} \cdot\|_{H^s}$ where
  $a_2 = \frac{1}{2}\sqrt{m^2 - \lambda_N^2}$ and denote by $\Sigma^s$  the corresponding spaces.
We set
\begin{align*}
\|\mathbf{u}\|_{\boldsymbol{\Sigma}^l}^2:=\|u_1\|_{\Sigma^l}^2+\|u_2\|_{\Sigma ^l}^2.
\end{align*}
Let $\boldsymbol{\Sigma}^\infty =\cap _{l\in \R }\boldsymbol{\Sigma}^l.$

\begin{proposition}\label{prop:rp}
There exist $\{\boldsymbol{\phi}_{\mathbf{m}}\}_{\mathbf{m}\in \mathbf{NR}}$   in $\boldsymbol{\Sigma}^\infty$,
 $\{\mathbf{G}_{\mathbf{m}}\}_{\mathbf{m}\in \mathbf{R}_{\mathrm{min}}} \subset\boldsymbol{\Sigma}^{\infty}$,
  $\{\lambda_{ \mathbf{n}j}\}_{\mathbf{n}\in \Lambda_0\cup \{ \mathbf{0}\}}\subset\R$  for $j=1,\cdots,N$ with $\boldsymbol{\phi}_{\mathbf{e}^j}=\boldsymbol{\Phi}_j$  and $\lambda_{ \mathbf{0} j}=\lambda_j$, a
   $\delta_1>0$ s.t.\ there exists $\widetilde{\mathbf{z}}_2 \in C^\infty(B_{\C^N}(0,\delta_1),\C^N)$ satisfying
\begin{align}\label{eq:boundz2}
	\|\widetilde{\mathbf{z}}_2\| _{\C ^N}\lesssim \sum_{\mathbf{m}\in 	\mathbf{R}_{\mathrm{min}}}|\mathbf{z}^{\mathbf{m}}|,
\end{align}
s.t. for any $l$
\begin{align}\label{eq:R1FRemainder}
\|\mathbf{R}[\mathbf{z}]\|_{\boldsymbol{\Sigma}^l}\lesssim  _l\|\mathbf{z}\|   _{\C ^N} \sum_{\mathbf{m}\in \mathbf{R}_{\mathrm{min}}}|\mathbf{z}^{\mathbf{m}}|,
\end{align}
where $\mathbf{R}[\mathbf{z}]$ is defined by the equality
\begin{align}\label{eq:rp}
D\boldsymbol{\phi}[\mathbf{z}]\widetilde{\mathbf{z}}=\mathbf{J}\(\mathbf{L}_1\boldsymbol{\phi}[\mathbf{z}]+\mathbf{f}[\boldsymbol{\phi}[\mathbf{z}]]-\sum_{\mathbf{m}\in \mathbf{R}_{\mathrm{min}}}\mathbf{z}^{\mathbf{m}}\mathbf{G}_{\mathbf{m}}-\mathbf{R}[\mathbf{z}]\),
\end{align}
 (where \eqref{eq:R1FRemainder}  and \eqref{eq:rp}  define the $\mathbf{G}_{\mathbf{m}}$)
  and
\begin{align}\label{eq:refprof1}
&\boldsymbol{\phi}[\mathbf{z}] :=\begin{pmatrix}
\phi_1[\mathbf{z}]\\ \phi_2[\mathbf{z}]
\end{pmatrix}=\sum_{\mathbf{m}\in \mathbf{NR}}\mathbf{z}^{\mathbf{m}}\boldsymbol{\phi}_{\mathbf{m}}, \\ &   \boldsymbol{\phi}_{\overline{\mathbf{m}}}=\overline{\boldsymbol{\phi}_{\mathbf{m}}}
    \label{eq:symmrp} \\ &  \label{eq:zexp}
\widetilde{\mathbf{z}} =\widetilde{\mathbf{z}}_0+\widetilde{\mathbf{z}}_1+\widetilde{\mathbf{z}}_2\ \text{with}\\ & \label{eq:z0}
\widetilde{\mathbf{z}}_0= (\im \lambda_1z_1,...,\im \lambda_{ N}z_{ N})=:\im  \boldsymbol{\lambda}  \mathbf{z} ,\\ & \label{eq:z1}
\widetilde{\mathbf{z}}_1
= (\im \sum_{\mathbf{n}\in \Lambda_0}\lambda_{ \mathbf{n}1}\mathbf{z}^{\mathbf{n}}z_1,...,\im
 \sum_{\mathbf{m}\in \Lambda_0} \lambda_{ \mathbf{n}N} \mathbf{z}^{\mathbf{n}} z_N),
 \\ & \boldsymbol{\lambda} _{ \overline{\mathbf{m}}} =  \boldsymbol{\lambda} _{ \mathbf{m}} \in \R ^{2N}\label{eq:lambdan}
\end{align}
where $\boldsymbol{\lambda} _{ \mathbf{m}} :=(\lambda_{ \mathbf{m}1},..., \lambda_{ \mathbf{m}N},-\lambda_{ \mathbf{m}1},..., -\lambda_{ \mathbf{m}N} ) $,
such that, setting
\begin{align}\label{eq:RForth}
\boldsymbol{\mathcal{H}}_{\mathrm{c}}[\mathbf{z}]:=\{\mathbf{u}\in \boldsymbol{\mathcal{H}}^1\ |\   \Omega(\mathbf{u},D_{\mathbf{z}}\boldsymbol{\phi}[\mathbf{z}]\mathbf{w})=0 \text{ for all }     \mathbf{w}\in \C^{ N} \}
\end{align}
and
\begin{align}\label{eq:rtilde}
\widetilde{\mathbf{R}}[\mathbf{z}]= \sum_{\mathbf{m}\in \mathbf{R}_{\mathrm{min}}}\mathbf{z}^{\mathbf{m}}\mathbf{G}_{\mathbf{m}}+\mathbf{R}[\mathbf{z}],
\end{align}
we have
\begin{align}
\mathbf{J}\widetilde{\mathbf{R}}[\mathbf{z}] \in \boldsymbol{\mathcal{H}}_{\mathrm{c}}[\mathbf{z}].\label{eq:R1FRemainder--}
\end{align}
\end{proposition}
 \proof    We begin observing    that $\mathbf{J}\mathbf{L}_1$  leaves the following decomposition invariant,  \begin{align}& \label{eq:Linz2eig13}   L^2 (\R , \C ^2 ) =   L^2_{discr} \oplus  L^2_{disp}  \text{  where }L^2_{discr}:= \oplus _{\lambda \in \sigma _p (\mathbf{J}\mathbf{L} _{1}) }  \ker \( \mathbf{L} _{1} -  \lambda \) ,
   \end{align}
   where $L^2_{disp} $   is   the $   \< \mathbf{J} \cdot , \cdot \>$--orthogonal  of $L^2_{discr}$.

\noindent  We insert \eqref{eq:refprof1} in \eqref{eq:rp}, using \eqref{eq:zexp}--\eqref{eq:z1}. We expand \begin{align*} &
 f (  {{\phi  }}_1 [ \mathbf{z} ] )=    \sum_{\ell =2}^M\frac{f^{( \ell )}(0  )}{\ell !}  {{\phi  }}_1 ^{\ell}[  \mathbf{z}] +  O(\|  \mathbf{ z} \| ^{M+1}),
\end{align*} Then,   for $  \mathbf{i}  ={ ^t(1,0)},$
\begin{align*} &   \sum_{\ell =2}^M\frac{f^{( \ell )}(0  )}{\ell !}  {{\phi  }}_1 ^{\ell}[  \mathbf{z}] \ \mathbf{i}
=  \sum_{\mathbf{m} \in \mathbf{NR}  }\mathbf{z} ^{\mathbf{m} }  \mathbf{h}_\mathbf{m}  +
\sum_{\substack{\mathbf{m} \in \mathbf{R} \cup \mathbf{I} \\  |\mathbf{m}|\le M} }\mathbf{z} ^{\mathbf{m} }
  \mathbf{h}_\mathbf{m}   +  O(\|  \mathbf{ z} \| ^{M+1})
\end{align*}
 where,   for   $ {\boldsymbol{\phi  }}  _{\mathbf{m}}   ={ ^t( {{\phi  }}_{1\mathbf{m}}  ,  {{\phi  }}_{2\mathbf{m}}  )}$,
\begin{align}\label{eq:gm}&
  \mathbf{h}_\mathbf{m}  =
  \sum_{\ell =2}^M\frac{f^{( \ell )}(0  )}{\ell !}
   \sum_{\substack{\mathbf{m}^1,\cdots,\mathbf{m}^\ell \in \mathbf{NR}\\ \mathbf{m}^1
   +\cdots+\mathbf{m}^{\ell}=\mathbf{m}}} {\phi}_{1 \mathbf{m}^1} \cdots {\phi}_{1\mathbf{m}^\ell} \   {\mathbf{i}} .
\end{align}
Using   \begin{align}\label{eq:difzm}
\( D_{\mathbf{z}}\mathbf{z}^{\mathbf{m}}\) \widetilde{\mathbf{z}}_0 =\im (\mathbf{m}\cdot \boldsymbol{\lambda}) \  \mathbf{z}^{\mathbf{m}}, \text{ where $\boldsymbol{\lambda}\mathbf{z} := ( \lambda _1z_1,..., \lambda _Nz_N)$,}
\end{align}
and recalling \eqref{eq:zexp},
we obtain
\begin{align*} &    D _{\mathbf{z}} \boldsymbol{\phi  } [ \mathbf{z}]     \widetilde{\mathbf{z}}[ \mathbf{z}] =    \im
\sum_{\mathbf{m} \in \mathbf{NR}  } (\mathbf{m}\cdot  \boldsymbol{\lambda  } )   \mathbf{z} ^{\mathbf{m} } \boldsymbol{\phi  }_{\mathbf{m} } + \im
\sum_{\substack{  \mathbf{m}  \in \mathbf{NR }, \ \mathbf{n} \in \boldsymbol{\Lambda  }_0    }}
 ( \mathbf{m}\cdot \boldsymbol{\lambda  } _{\mathbf{n}} ) \mathbf{z}^{\mathbf{n}}    \mathbf{z}^{\mathbf{m}}\boldsymbol{\phi  }_{\mathbf{m} }
   + D _{\mathbf{z}} \boldsymbol{\phi  } [ \mathbf{z}]   \widetilde{\mathbf{z}} _{2} .
\end{align*}
Let us set  \begin{align*}
 {\boldsymbol{ \mathcal{R}}}  [\mathbf{z}]&:=  \mathbf{J}\(\mathbf{L}_1\boldsymbol{\phi}[\mathbf{z}]+\mathbf{f}[\boldsymbol{\phi}[\mathbf{z}]]  \)
 -D_{\mathbf{z}}\boldsymbol{\phi}[\mathbf{z}]( \widetilde{\mathbf{z}}- \widetilde{\mathbf{z}}_2 ) \\& =
  \mathbf{J} \begin{pmatrix}
L_1 \phi_1[\mathbf{z}]+f(\phi_1[\mathbf{z}])\\ \phi_2[\mathbf{z}]
\end{pmatrix}
-D_{\mathbf{z}}\boldsymbol{\phi}[\mathbf{z}]( \widetilde{\mathbf{z}}- \widetilde{\mathbf{z}}_2 ) .
\end{align*}
We expand now to get
\begin{align}\label{eq:RFtildeexpan} &   {\boldsymbol{ \mathcal{R}}}[\mathbf{z}] =  \sum_{\mathbf{m} \in \mathbf{NR}  }
\mathbf{z} ^{\mathbf{m}}  { \mathcal{R}}_{\mathbf{m}} +   \sum_{\substack{\mathbf{m} \in \mathbf{R} \cup \mathbf{I} \\
 |\mathbf{m}|\le M} }\mathbf{z} ^{\mathbf{m} }  \mathcal{R}_{\mathbf{m}}   +  O(\|  \mathbf{ z} \| ^{M+1}),
\end{align}
where
\begin{align*} &   { \mathcal{R}}_{\mathbf{m}} =
 \(  \mathbf{J}\mathbf{L} _{1}   - {\im}  \boldsymbol{\lambda  } \cdot \mathbf{m}  \)    \boldsymbol{\phi  }_{\mathbf{m} } + \mathcal{E}_\mathbf{m}
 \text{  where }  \\&   \mathcal{E}_\mathbf{m} = \mathbf{J}\mathbf{h}_\mathbf{m}   -
  \sum_{\substack{\mathbf{m}' +\mathbf{n}'=\mathbf{m}\\ \mathbf{m}' \in \mathbf{NR }, \ \mathbf{n}'\in \boldsymbol{\Lambda  }_0   }}
      \im  (\boldsymbol{\lambda  }_{\mathbf{n}'} \cdot \mathbf{m}' )  \boldsymbol{\phi  }_{\mathbf{m}'   }            .
\end{align*}
We seek $ { \mathcal{R}}_{\mathbf{m}} \equiv 0$   for $\mathbf{m} \in \mathbf{NR} $. For $\|\mathbf{m}\|=1$ the equation reduces to
$\(  \mathbf{J}\mathbf{L} _{1}    -\im \boldsymbol{\lambda  } \cdot \mathbf{m}  \)    \boldsymbol{\phi  }_{\mathbf{m} }  =0$, so that we can set $\boldsymbol{\phi  }_{\mathbf{e}^j } = \boldsymbol{\Phi  }_{ j }   $   and  $\boldsymbol{\phi  }_{\overline{\mathbf{e} }^j } = \overline{\boldsymbol{\Phi  }}_{ j }   $.  Let us consider now $\|\mathbf{m}\|\ge 2 $   with $\mathbf{m}\not \in  \cup _{j=1}^{N}\(    \boldsymbol{\Lambda  }_j\cup   \overline{\boldsymbol{\Lambda  }}_j \) $. In this case, let us assume by induction that  $\boldsymbol{\phi  }_{\mathbf{m} '}$ and $ \boldsymbol{\lambda  }_{\mathbf{m} '}$
have been defined for $\|\mathbf{m}' \| < \| \mathbf{m} \| $ and that they satisfy \eqref{eq:symmrp}--\eqref{eq:lambdan}.  Then, from \eqref{eq:gm} we obtain $h_{\overline{\mathbf{m}}}=\overline{h}_\mathbf{m}$ and $\mathcal{E}_{\overline{\mathbf{m}}}=\overline{\mathcal{E}}_\mathbf{m}$.
 We can solve  $  \mathcal{R}_{\mathbf{m} }= 0$  writing  $\boldsymbol{\phi  }_{\mathbf{m}   }=\( \mathbf{J} \mathbf{L} _{1}    - \im \boldsymbol{\lambda  } \cdot \mathbf{m}  \) ^{-1}\mathcal{E}_\mathbf{m}$.  By  $\boldsymbol{\lambda  } \cdot \overline{\mathbf{m}}=-\boldsymbol{\lambda  } \cdot \mathbf{m}$, we conclude
$\boldsymbol{\phi  }_{\overline{\mathbf{m}}}=\overline{\boldsymbol{\phi  }}_\mathbf{m}$.

We now consider $\mathbf{m}\in  \boldsymbol{\Lambda  }_j$. We assume by induction that $\boldsymbol{\phi  }_{\mathbf{m} '} $ have been defined for $\|\mathbf{m}' \| < \| \mathbf{m} \| $ and  so too $ \boldsymbol{\lambda  }_{\mathbf{n} '}$
  for $\|\mathbf{n}' \| <\| \mathbf{m} \| -1$. Then, for $ \mathbf{m}=\mathbf{n}+\mathbf{e}^j$ where $\mathbf{n}\in  \boldsymbol{\Lambda  }_0$,
   $ { \mathcal{R}}_{\mathbf{m}}= 0$  becomes
\begin{align} &    \nonumber     \( \mathbf{J} \mathbf{L} _{1}     - \im   {\lambda  } _j  \)    \boldsymbol{\phi  }_{\mathbf{m} }  =\mathcal{E}_\mathbf{m}= \im  \boldsymbol{\lambda  }_{\mathbf{n} } \cdot \mathbf{e}^{j} \boldsymbol{\Phi  }_{j }  -  \mathcal{K}_\mathbf{m}  \text{  with } \\&  \mathcal{K}_\mathbf{m}:= \mathbf{J} h_\mathbf{m}- \sum_{\substack{\mathbf{m}' +\mathbf{n}'=\mathbf{m}\\ \mathbf{m}' \in \mathbf{NR  }, |\mathbf{m}'|\ge 2, \ \mathbf{n}'\in \boldsymbol{\Lambda  }_0   }}     {\im}  \boldsymbol{\lambda  }_{\mathbf{n}'} \cdot \mathbf{m}'   \boldsymbol{\phi  }_{\mathbf{m}'   }   .\label{eq:lambdaj}
\end{align}
This equation can be solved if we impose $ \( \mathbf{J}  {\mathcal{E}}_\mathbf{m} ,\overline{\boldsymbol{\Phi  }}_{j } \) =0 $, that is, for  $ {\lambda   }_{\mathbf{n} j } :=\boldsymbol{\lambda  }_{\mathbf{n} } \cdot \mathbf{e}^{j}$, if
\begin{align*} & - {\im}  {\lambda   }_{\mathbf{n} j }   \( \mathbf{J}  \boldsymbol{\Phi  }_{j } ,\overline{\boldsymbol{\Phi  }}_{j }   \) =-    2   {\lambda   }_{\mathbf{n} j }   \lambda _j     =      \( \mathbf{J}  {\mathcal{K}}_\mathbf{m} ,\overline{\boldsymbol{\Phi  }}_{j }\)   ,
\end{align*}
 which is true for ${\lambda   }_{\mathbf{n} j }        =    -2^{-1} \lambda _j ^{-1} \( \mathbf{J}  {\mathcal{K}}_\mathbf{m} ,\overline{\boldsymbol{\Phi  }}_{j }\)  $.   Then we can solve for $ \boldsymbol{\phi  }_{\mathbf{m} } = -\(   \mathbf{J}\mathbf{L} _{1}   - \im  {\lambda  } _j  \) ^{-1} \mathcal{K}_\mathbf{m} $     in the complement, in \eqref{eq:Linz2eig13}, of $\ker (  \mathbf{J}  \mathbf{L} _{1} -\im  \lambda _j) $.

\noindent We want to show that ${\lambda   }_{\mathbf{n} j } \in \R $. For the corresponding $\overline{\mathbf{m}} \in \overline{\boldsymbol{\Lambda  }}_{j }$,
we have
\begin{align} &        \(  \mathbf{J}\mathbf{L} _{1}  + \im  {\lambda  } _j  \)    \boldsymbol{\phi  }_{\overline{\mathbf{m}} }  = \im  \boldsymbol{\lambda  }_{\mathbf{n} } \cdot \overline{\mathbf{e}}^{j} \overline{\boldsymbol{\Phi  }}_{j }  -  \mathcal{K}  _{ \overline{\mathbf{m}} }     \text{  with } \nonumber \\&  \mathcal{K}_  { \overline{\mathbf{ {m}}} }:=  \mathbf{J}h_ { \overline{\mathbf{ {m}}} }- \sum_{\substack{\overline{\mathbf{m}}' +\mathbf{n}'=\overline{\mathbf{m}}\\ \mathbf{m}' \in \mathbf{NR_2 }, \ \mathbf{n}'\in \boldsymbol{\Lambda  }_0   }}   \im  \boldsymbol{\lambda  }_{\mathbf{n}'} \cdot \overline{\mathbf{m}}'   \boldsymbol{\phi  }_{\overline{\mathbf{m}}'   }     . \label{eq:barlambdaj}
\end{align}
Notice that by induction  $\mathcal{K}  _{ \overline{\mathbf{m}} }=\overline{\mathcal{K} } _{ \mathbf{m} }$. Since  $\boldsymbol{\lambda  }_{\mathbf{n} } \cdot \overline{\mathbf{e}}^{j} = - {\lambda  }_{\mathbf{n}j }$,  taking the complex conjugate of  \eqref{eq:lambdaj} we obtain
\begin{equation}\label{eq:bothlambdaj}
\begin{aligned} &  \( \mathbf{J} \mathbf{L} _{1}  +  \im  {\lambda  } _j  \)    \boldsymbol{\phi  }_{\overline{\mathbf{m}} }  = \im   {\lambda  }_{\mathbf{n} j}   \overline{\boldsymbol{\Phi  }}_{j }  -  \overline{\mathcal{K}}  _{ \mathbf{m}  } \text{  and  }\\& \( \mathbf{J} \mathbf{L} _{1}   +  \im   {\lambda  } _j  \)    \overline{\boldsymbol{\phi  }}_{ \mathbf{m}  }  =  \im   \overline{{\lambda  }}_{\mathbf{n} j}  \overline{ \boldsymbol{\Phi  }} _{j }  -  \overline{\mathcal{K}}  _{ \mathbf{m}  } .
\end{aligned}
\end{equation}
Applying $  \( \mathbf{J} \cdot ,    \boldsymbol{\Phi  } _{j }\) $  on both the last two equations,  we obtain
\begin{align*} &      \text{$\im  {\lambda  }_{\mathbf{n} j}  \( \mathbf{J}   \overline{\boldsymbol{\Phi   }}_{j } ,   \boldsymbol{\Phi   }_{j } \) =  \( \mathbf{J}  \overline{\mathcal{K}}  _{ \mathbf{m}  } ,    \boldsymbol{\Phi   }_{j }\)  $   and   $\im   \overline{{\lambda  }}_{\mathbf{n} j}    \( \mathbf{J}  \overline{ \boldsymbol{\Phi  }}_{j } ,    \boldsymbol{\Phi   }_{j }\)  =  \( \mathbf{J}    \overline{\mathcal{K}}  _{ \mathbf{m}  } ,    \boldsymbol{\Phi   }_{j }\) $.}
\end{align*}
Hence ${\lambda  }_{\mathbf{n} j}=  \overline{{\lambda  }}_{\mathbf{n} j} $ and we have proved that ${\lambda  }_{\mathbf{n} j} \in \R$.

\noindent Since the equations in \eqref{eq:bothlambdaj} are the same, we conclude $ \boldsymbol{\phi  }_{\overline{\mathbf{m}} } =  \overline{\boldsymbol{\phi  }}_{ \mathbf{m}  } $.

\noindent  We consider now
\begin{equation}\label{eq:defGm}
\begin{aligned} &    \mathbf{J}  \widetilde{\mathbf{R}}[\mathbf{z}] = \boldsymbol{ \mathcal{R}} [\mathbf{z}] -D_{\mathbf{z}}\boldsymbol{\phi}[\mathbf{z}]\widetilde{\mathbf{z}}_{2},
\end{aligned}
\end{equation}
 where we seek $\widetilde{\mathbf{z}}_{2}$  so that \eqref{eq:R1FRemainder--} is true.
This will follow from (here $\mathbf{J}^{-1}=-\mathbf{J}$)
\begin{align*} \< \mathbf{J}  \boldsymbol{ \mathcal{R}} [\mathbf{z}]   ,D_{\mathbf{z}}\boldsymbol{\phi}[\mathbf{z}] \mathbf{w}\> -\<\mathbf{J} D_{\mathbf{z}}\boldsymbol{\phi}[\mathbf{z}]\widetilde{\mathbf{z}}_{2},D_{\mathbf{z}}\boldsymbol{\phi}[\mathbf{z}]\mathbf{w}\>=0  \text{ for the standard basis $\mathbf{w} = e_1, \im e_1,...,e_N, \im e_N$}.\end{align*}
Since the restriction of  $\< \mathbf{J}\cdot , \cdot \>$ in $L^2_{discr}$ is a non--degenerate symplectic form and from  $\boldsymbol{\phi  }_{\mathbf{e}^j } = \boldsymbol{\Phi  }_{ j }   $   and  $\boldsymbol{\phi  }_{\overline{\mathbf{e} }^j } = \overline{\boldsymbol{\Phi  }}_{ j }
$,
the Implicit Function  Theorem guarantees the existence of $\widetilde{\mathbf{z}} _{  2} \in C ^{\infty}( \mathcal{B} _{\C ^N} (0, \delta _1), \C^N)$  with $\widetilde{\mathbf{z}} _{  2}(\mathbf{0})=\mathbf{0}$ for a sufficiently small $\delta _1>0$. Furthermore, from the last formula and from the fact that in the expansion \eqref{eq:RFtildeexpan} we have  $ \mathcal{{R}} _{\mathbf{m}}=0$  for all $\mathbf{m} \in \mathbf{NR}$, we obtain  the bound \eqref{eq:R1FRemainder}.

\noindent Solving in \eqref{eq:defGm}  for $\widetilde{\mathbf{R}}[\mathbf{z}] = \mathbf{J}^{-1}\boldsymbol{ \mathcal{R}} [\mathbf{z}] -\mathbf{J}^{-1}D_{\mathbf{z}}\boldsymbol{\phi}[\mathbf{z}]\widetilde{\mathbf{z}}_{2}$,  exploiting the fact that we have ${ \mathcal{R}}_{\mathbf{m}}$  for all $\mathbf{m} \in \mathbf{NR}$ and by   \eqref{eq:boundz2}, by Taylor expansion  in the variable $\mathbf{z}$,  we obtain   expansion \eqref{eq:rtilde}, with the estimate \eqref{eq:R1FRemainder}.

\qed

 We assume the following.
\begin{assumption}[Fermi Golden Rule]\label{ass:FGR}
For any $\mathbf{m}\in \mathbf{R}_{\mathrm{min}}$, there exists a bounded solution $\mathbf{g}_{\mathbf{m}}$ of $\mathbf{J}\mathbf{L}_1 \mathbf{g}_{\mathbf{m}}=\im (\mathbf{m}\cdot\boldsymbol{\lambda}) \mathbf{g}_{\mathbf{m}}$ s.t.\
\begin{align}\label{ass:FGR1}
\<\mathbf{G}_{\mathbf{m}},\mathbf{g}_{\mathbf{m}}\>=\gamma_{\mathbf{m}}> 0.
\end{align}
\end{assumption}

\begin{remark} Notice that all it matters in \eqref{ass:FGR1} is to have $\gamma_{\mathbf{m}}\neq 0$, since    by replacing $\mathbf{g}_{\mathbf{m}}$  with $-\mathbf{g}_{\mathbf{m}}$, we can then obtain $\gamma_{\mathbf{m}}>0$.

\end{remark}

Recall now from  Sect.\ 3 Deift-Trubowitz \cite{DT79CPAM},   the following result on Darboux transformations, here stated with stricter hypotheses than in  \cite{DT79CPAM}.

\begin{proposition}\label{prop:Darboux}
Let $W\in \mathcal{S}(\R,\R)$ s.t $\sigma_{\mathrm{d}}(-\partial_x^2+W)\neq \emptyset$ and let $ \omega =\inf \sigma_{\mathrm{d}}(-\partial_x^2+W)$.
Let $\psi$ be a ground state  of $-\partial_x^2+W$, that is a generator of $\ker \(-\partial_x^2+W-\omega \)$,
and set $A_W=\frac{1}{\psi}\partial_x\(\psi \cdot\)$   (recall that $\psi (x)\neq 0$ for all $x\in \R$).
Then, there exists $W_1\in \mathcal{S}(\R,\R)$ s.t.
\begin{align*}
A_WA_W^*=-\partial_x^2+W-\omega,\ A_W^*A_W=-\partial_x^2 +W_1 -\omega
\end{align*}
 and $\sigma_{\mathrm{d}}(-\partial_x^2+W_1)=\sigma_{\mathrm{d}}(-\partial_x^2+W)\setminus\{\omega\}$.
\end{proposition} \qed

Using Proposition \ref{prop:Darboux}, we inductively define $V_j \in \mathcal{S}(\R,\R)$ ($j=1,\cdots,N+1$) by
\begin{enumerate}
\item $V_1:=V$, $L_1:=-\partial_x^2+V_1+m^2$, $\psi_1=\phi_1$ and $A_1=A_{V_1}$.
\item Given $V_k$, we define
\begin{align}\label{def:Ak}
A_k:=A_{V_k}\text{ and  }L_{k+1}:=-\partial_x^2+V_{k+1} +m^2:=A_k^*A_k+ \lambda ^2_k,
\end{align}
and, by Proposition \ref{prop:Darboux}, we have $L_k =-\partial_x^2+V_{k } +m^2 =A_k A_k^* + \lambda ^2_k$
\end{enumerate}
From Proposition \ref{prop:Darboux}, we have
\begin{align*}
\sigma_{\mathrm{d}}(L_{k })=\{  \lambda ^2_j\ |\ j=k,\cdots,N\},\ k=1,\cdots,N,\ \text{and}\ \sigma_{\mathrm{d}}(L_{N+1})=\emptyset.
\end{align*}
If $\psi_k$ is the ground state of $L_k$  and  $A_k=\frac{1}{\psi_k}\partial_x\(\psi_k\cdot\)$ then,
from
\begin{align}\label{eq:DarConj1}
A_j^* L_j=A_j^*(A_jA_j^*+\lambda ^2_j)=(A_j^*A_j+\lambda ^2_j)A_j^*=L_{j+1}A_j^* ,
\end{align}
we have the conjugation relation
\begin{align}\label{eq:DarConj2}
\mathcal{A}^* L_1=L_{N+1}\mathcal{A}^*,
\end{align}
where
\begin{align}\label{def:calA}
\mathcal{A}=A_1\cdots A_N \text{  and }\mathcal{A}^*=A_N^*\cdots A_1^*.
\end{align}
 We write $L_D:=L_{N+1}$     and  $V_D:=V_{N+1}$. We assume that $V_D$ is repulsive with respect to the origin, specifically the following.
\begin{assumption}\label{ass:repulsive}
We assume $xV_{\mathrm{D}}'\leq 0$ and $xV_{\mathrm{D}}'(x)\not\equiv 0$.
\end{assumption}

The main  point for us is that $L_1$ has eigenvalues, we have the orthogonal decomposition
\begin{align}\label{def:orthecomp}
 L^2(\R , \C ) = \(   \oplus  _{j=1} ^{N} \ker  \(  L_1-\lambda _j ^2\) \) \oplus L ^{2}_{c}(L_1),
\end{align}
where $ L ^{2}_{c}(L_1)$  is the continuous spectrum component associated to $L_1$. We denote by $P_c$ the orthogonal projection onto $L ^{2}_{c}(L_1)$.

\section{Main estimates and proof of Theorem \ref{thm:main}.} \label{sec:modul}

Using the refined profile given in Proposition \ref{prop:rp}, we first decompose the solution by appropriate orthogonality condition.

\begin{lemma} [Modulation]\label{lem:modulat}
There exists $\delta_1>0$ s.t.\ there exists $\mathbf{z}\in C^\infty(B_{\boldsymbol{\mathcal{H}}^{1}}(0,\delta_1),\C^{N})$ s.t. $\mathbf{z}(\mathbf{0})=\mathbf{0}$ and
\begin{align}
\boldsymbol{\eta}[\mathbf{u}]:=\mathbf{u}-\boldsymbol{\phi}[\mathbf{z}(\mathbf{u})]\in \boldsymbol{\mathcal{H}}_{\mathrm{c}}[\mathbf{z}(\boldsymbol{u})]. \label{eq:lem:mod-1}
\end{align}
Furthermore, we have
\begin{align}
\|\mathbf{u}\|_{\boldsymbol{\mathcal{H}}^1} \sim \|\boldsymbol{\eta}[\mathbf{u}]\|_{\boldsymbol{\mathcal{H}}^1}+\|\mathbf{z}(\mathbf{u})\| .   \label{eq:lem:mod-2}
\end{align}
\end{lemma}

\begin{proof}
Standard.
\end{proof}

In the following, we fix a solution $\mathbf{u}$ of \eqref{NLKG} with $\mathbf{u}(0)=\mathbf{u}_0$ satisfying the assumption of Theorem \ref{thm:main} (with $\delta_0>0$ to be determined).
We write $\mathbf{z}(t)=\mathbf{z}(\mathbf{u}(t))$ and $\boldsymbol{\eta}(t)=\boldsymbol{\eta}[\mathbf{u}(t)]$.
By the conservation of energy and by \eqref{eq:lem:mod-2} we have
\begin{align}\label{eq:orbbound}
\|\mathbf{z}\|_{L^\infty   (\R ,\C^N)} + \|\boldsymbol{\eta}\|_{L^\infty (\R ,\boldsymbol{\mathcal{H}}^{1})}\lesssim \delta.
\end{align}

Substituting $\mathbf{u}=\boldsymbol{\phi}[\mathbf{z}]+\boldsymbol{\eta}$ into \eqref{NLKG}, we have
\begin{align}\label{eq:modeq}
\dot{\boldsymbol{\eta}}+D_{\mathbf{z}}\boldsymbol{\phi}[\mathbf{z}](\dot{\mathbf{z}}-\widetilde{\mathbf{z}})=\mathbf{J}
\(\mathbf{L}[\mathbf{z}]\boldsymbol{\eta}+\mathbf{F}[\mathbf{z},\boldsymbol{\eta}]+\sum_{\mathbf{m}\in \mathbf{R}_{\mathrm{min}}}\mathbf{z}^{\mathbf{m}}\mathbf{G}_{\mathbf{m}}+\mathbf{R}[\mathbf{z}]\),
\end{align}
where, for $d\mathbf{f}$ the Frech\'et derivative  of $\mathbf{f}$,
\begin{align}\label{eq:defLz}
\mathbf{L}[\mathbf{z}]&=\mathbf{L}_1+d\mathbf{f}[\boldsymbol{\phi}[\mathbf{z}]],\\ \label{eq:defF}
\mathbf{F}[\mathbf{z},\boldsymbol{\eta}]&=\mathbf{f}[\boldsymbol{\phi}[\mathbf{z}]+\boldsymbol{\eta}]-\mathbf{f}[\boldsymbol{\phi}[\mathbf{z}]]-d\mathbf{f}[\boldsymbol{\phi}[\mathbf{z}]]\boldsymbol{\eta}.
\end{align}
Notice that $\mathbf{F}[\mathbf{z},\boldsymbol{\eta}]={}^t(F_1[\mathbf{z},\eta_1]\ 0)$ where
\begin{align}  F_1[\mathbf{z},\eta_1]= f( {\phi}_1[\mathbf{z}]+ {\eta} _1)-f( {\phi}_1[\mathbf{z}] )-f'( {\phi}_1[\mathbf{z}] ){\eta} _1.   \label{eq:defF1}
 \end{align}

\noindent We will consider constants  $A, B,\varepsilon >0$ satisfying
 \begin{align}\label{eq:relABg}
\log(\delta ^{-1})\gg\log(\epsilon ^{-1})\gg  A\gg    B^2\gg B \gg  \exp \( \varepsilon ^{-1} \) \gg 1.
 \end{align}
 We will denote by    $o_{\varepsilon}(1)$  constants depending on $\varepsilon$ such that
 \begin{align}\label{eq:smallo}
 \text{ $o_{\varepsilon}(1) \xrightarrow {\varepsilon  \to 0^+   }0.$}
 \end{align}
 Let     \begin{align} \label{eq:kappa}
 \kappa \in \(0,\min(m-\lambda_N,a_1)/10   \).
 \end{align}
 We will consider the norms
\begin{align}\label{eq:normA}&
\| \boldsymbol{\eta} \|_{ \boldsymbol{ \Sigma }_A} :=\left \| \sech \(\frac{2}{A} x\) \eta_1'\right \|_{L^2} +A^{-1}\left \|    \sech \(\frac{2}{A} x\) \boldsymbol{\eta}\right\|_{L^2}  \text{ and}\\&  \| \boldsymbol{\eta} \|_{ \boldsymbol{L }^2_{-\kappa} } :=\left \| \sech \( \kappa  x\)  \boldsymbol{\eta}\right \|_{L^2}.\label{eq:normk}
\end{align}
We will prove the following  continuation argument.
 \begin{proposition}\label{prop:contreform}   Under the assumptions  \ref{ass:generic},   \ref{ass:FGR}  and \ref{ass:repulsive},
  for any small $\epsilon>0 $
there exists  a    $\delta _0 = \delta _0(\epsilon )   $ s.t.\  if    in  $I=[0,T]$ we have
\begin{align}&
\|\dot {\mathbf{z}} - \widetilde{\mathbf{z}}\|_{L^2(I)}+\sum_{\mathbf{m}\in \mathbf{R}_{\mathrm{min}}}\|\mathbf{z}^\mathbf{m}\|_{L^2(I)}+ \|   \boldsymbol{\eta}  \|_{L^2(I, \boldsymbol{ \Sigma }_A  \cap   \boldsymbol{L }^2_{-\kappa})}\le   \epsilon   \label{eq:main11}
\end{align}
then  for $\delta  \in (0, \delta _0)$  and $\delta =\|\mathbf{u}_0\|_{\boldsymbol{\mathcal{H}}^1}$
     inequality   \eqref{eq:main11} holds   for   $\epsilon$ replaced by $ o_{\varepsilon}(1)   \epsilon $    where $o_{\varepsilon}(1) \xrightarrow {\varepsilon  \to 0^+   }0 $.
\end{proposition}
Notice that Proposition \ref{prop:contreform} implies by standard continuation arguments Theorem \ref{thm:main}.

We will prove    Proposition \ref{prop:contreform}  from the following statements.

\begin{proposition}\label{prop:modp}
We have
\begin{align}
\|\dot{\mathbf{z}}-\widetilde{\mathbf{z}}\|_{L^2(I)}=
o_{\varepsilon}(1) \|  \boldsymbol{\eta}  \|_{L^2(I,  \boldsymbol{L }^2_{-\kappa})}      . \label{eq:lem:estdtz}
\end{align}
\end{proposition}

 \begin{proposition}[FGR estimate]\label{prop:FGR}
 We have
 \begin{align}\label{eq:FGRint}
 \sum_{\mathbf{m}\in \mathbf{R}_{\mathrm{min}}}\|\mathbf{z}^{\mathbf{m}}\|_{L^2(I)}\lesssim \delta +A^{-1/4} \|   \boldsymbol{\eta}  \|_{L^2(I, \boldsymbol{ \Sigma }_A  )}  .
 \end{align}
 \end{proposition}

\begin{proposition}[1st virial estimate]\label{prop:1stvirial}
We have
\begin{align}
 \|   \boldsymbol{\eta}  \|_{L^2(I, \boldsymbol{ \Sigma }_A  )}  \lesssim \delta +  \|  \boldsymbol{\eta}  \|_{L^2(I,  \boldsymbol{L }^2_{-\kappa})}  +\sum_{\mathbf{m}\in 	\mathbf{R}_{\mathrm{min}}}\|\mathbf{z}^{\mathbf{m}}\|_{L^2} .\label{eq:1stvInt}
\end{align}

\end{proposition}

\begin{proposition}[2nd virial  estimate]\label{prop:2ndvirial}
We have
\begin{align}\label{eq:2ndv}
 \|  \boldsymbol{\eta}  \|_{L^2(I,  \boldsymbol{L }^2_{-\kappa})}  \lesssim  B\varepsilon^{-N}\delta + A ^{-1/4}  \|   \boldsymbol{\eta}  \|_{L^2(I, \boldsymbol{ \Sigma }_A  )}  +\sum_{\mathbf{m}\in 	\mathbf{R}_{\mathrm{min}}}\|\mathbf{z}^{\mathbf{m}}\|_{L^2} .
\end{align}
\end{proposition}

  {\it Proof of Proposition \ref{prop:contreform} assuming Propositions
  \ref{prop:modp}--\ref{prop:2ndvirial}.}
  From  Propositions \ref{prop:modp} and \ref{prop:FGR}  and from \eqref{eq:relABg} we have
  \begin{align}
\|\dot{\mathbf{z}}-\widetilde{\mathbf{z}}\|_{L^2(I)} +\sum_{\mathbf{m}\in \mathbf{R}_{\mathrm{min}} }\|\mathbf{z}^{\mathbf{m}}\|_{L^2(I)}= o_\varepsilon(1) \epsilon .\label{FGReqfinal}
\end{align}
  Entering this in \eqref{eq:2ndv} we get
\begin{align}\label{eq:2ndvfin}
\|  \boldsymbol{\eta}  \|_{L^2(I,  \boldsymbol{L }^2_{-\kappa})}=     o_\varepsilon(1) \epsilon .
\end{align}
   Entering \eqref{FGReqfinal} and \eqref{eq:2ndvfin} in \eqref{eq:1stvInt} we get $\|   \boldsymbol{\eta}  \|_{L^2(I, \boldsymbol{ \Sigma }_A  )}= o_\varepsilon(1) \epsilon .$  This completes the proof of  Proposition \ref{prop:contreform}.

   \qed

\textit{Proof of Theorem \ref{thm:main}.} By continuity,  Proposition \ref{prop:contreform} implies that inequality \eqref{eq:main11} is valid with $I=\R _+$. This implies \eqref{eq:main2} (adjusting $\epsilon$). From the equation for $\mathbf{z}$, see  \eqref{eq:eqz} below, we have $\dot {\mathbf{z}}\in L^\infty (\R , \C ^N )$. By $\mathbf{z}^{\mathbf{m}}\in L^2 (\R )$
for any $\mathbf{m} \in \mathbf{R} _{\min}$, so in particular $z_j ^{m_j}\in L^2 (\R )$ for $m_j$ the largest $m_j\in \N$ such that $(m_j-1) \lambda _j<m$, we have
$\displaystyle \lim _{t\to  +\infty}\mathbf{z}(t)=0$. \qed

 \section{Proof of Proposition \ref{prop:modp}}
\label{sec:propdmodes1}

 \textit{ Proof of Proposition \ref{prop:modp}.}
 We fix an even function $\chi\in C_0^\infty(\R , [0,1])$ satisfying
\begin{align}  \label{eq:chi} \text{$1_{[-1,1]}\leq \chi \leq 1_{[-2,2]}$ and $x\chi'(x)\leq 0$ and set $\chi_A:=\chi(\cdot/A)$.}
\end{align}

\begin{lemma}\label{lem:estF}  For    the $F_1$   in  \eqref{eq:defF1},
we have
\begin{align}
\|   \sech \(  \kappa x\)    F_1[\mathbf{z},\boldsymbol{\eta}]\|_{L^2}&\lesssim \delta \|       \sech \(  \kappa x\) \eta_1\|_{L^2}, \label{eq:lem:estF1} \\
\|\chi_A F_1[\mathbf{z},\boldsymbol{\eta}]\|_{L^1}&\lesssim A^{1/2}\delta \|      \sech \(   \frac{2}{A}x\)     \eta_1\|_{L^2}.    \label{eq:lem:estF2}
\end{align}
\end{lemma}

\begin{proof}
By Taylor expansion,
$F_1[\mathbf{z},\boldsymbol{\eta}]=\int_0^1(1-t)f''(\phi_1[\mathbf{z}]+t\eta_1)\eta_1^2\,dt$.
Thus,
\begin{align*}
	\|     \sech \(  \kappa x\)F_1[\mathbf{z},\boldsymbol{\eta}]\|_{L^2} &\lesssim \sup_{|u|\leq 1} | f''(u)| \|\eta_1\|_{L^\infty}\|     \sech \(  \kappa x\)\eta_1\|_{L^2}\lesssim \delta \|      \sech \(  \kappa x\)\eta_1\|_{L^2},\\
\|\chi_A F_1[\mathbf{z},\boldsymbol{\eta}]\|_{L^1}&\lesssim \sup_{|u|\leq 1} | f''(u)| \|\eta_1\|_{L^\infty}\|\eta_1 \chi_A\|_{L^1}\lesssim A^{1/2}\delta\| \sech \(   \frac{2}{A}x\) \eta_1\|_{L^2}^2,
\end{align*}
where we have used $  \sech \(   \frac{2}{A}x\) \sim 1$ in $\mathrm{supp} \chi_A$, \eqref{eq:orbbound} and the embedding $H^1(\R )\hookrightarrow L^\infty (\R )$.
\end{proof}

\begin{lemma}\label{lem:modbound}
We have
\begin{align}\label{lem:modbound1}
\|\dot{\mathbf{z}}-\widetilde{\mathbf{z}}\|   \lesssim \delta   \|     \sech \(  \kappa x\) \boldsymbol{\eta}\|_{L^2}.
\end{align}
\end{lemma}

\begin{proof} Recalling \eqref{eq:rtilde} and \eqref{eq:defLz},
differentiating \eqref{eq:rp} we have  for $\mathbf{w}\in \C^N$
\begin{align*}
D^2_{\mathbf{z}}\boldsymbol{\phi}[\mathbf{z}](\widetilde{\mathbf{z}},\mathbf{w})+D_{\mathbf{z}}
\boldsymbol{\phi}[\mathbf{z}]D_{\mathbf{z}}\widetilde{\mathbf{z}}(\mathbf{z})\mathbf{w}+\mathbf{J}D  _{\mathbf{z}}   \widetilde{\mathbf{R}}[\mathbf{z}] \mathbf{w}=\mathbf{J}\mathbf{L}[\mathbf{z}]D_{\mathbf{z}}\boldsymbol{\phi}[\mathbf{z}]\mathbf{w}.
\end{align*}
We apply  $\Omega( \cdot ,D_{\mathbf{z}}\boldsymbol{\phi}[\mathbf{z}]\mathbf{w})$  to  \eqref{eq:modeq}, obtaining
\begin{align*} & \Omega( \dot{\boldsymbol{\eta}} ,D_{\mathbf{z}}\boldsymbol{\phi}[\mathbf{z}]\mathbf{w})  + \Omega( D_{\mathbf{z}}\boldsymbol{\phi}[\mathbf{z}](\dot{\mathbf{z}}-\widetilde{\mathbf{z}}) ,D_{\mathbf{z}}\boldsymbol{\phi}[\mathbf{z}]\mathbf{w})
\\& = \< \mathbf{L}[\mathbf{z}]\boldsymbol{\eta}, D_{\mathbf{z}}\boldsymbol{\phi}[\mathbf{z}]\mathbf{w}\> + \< \mathbf{F}[\mathbf{z},\boldsymbol{\eta}], D_{\mathbf{z}}\boldsymbol{\phi}[\mathbf{z}]\mathbf{w}\>  ,
\end{align*}
where we used $\Omega( \mathbf{J} \widetilde{\mathbf{R}}[\mathbf{z}] ,D_{\mathbf{z}}\boldsymbol{\phi}[\mathbf{z}]\mathbf{w})=0, $ that is \eqref{eq:R1FRemainder--}.
 Using $\boldsymbol{\eta}\in \boldsymbol{\mathcal{H}}_{\mathrm{c}}[\mathbf{z} ]$, we have
 \begin{align*} \< \mathbf{L}[\mathbf{z}]\boldsymbol{\eta}, D_{\mathbf{z}}\boldsymbol{\phi}[\mathbf{z}]\mathbf{w}\>  &= \< \boldsymbol{\eta}, \mathbf{L}[\mathbf{z}] D_{\mathbf{z}}\boldsymbol{\phi}[\mathbf{z}]\mathbf{w}\>   = \< \boldsymbol{\eta}, \mathbf{J}^{-1}D^2_{\mathbf{z}}\boldsymbol{\phi}[\mathbf{z}](\widetilde{\mathbf{z}},\mathbf{w})+D  _{\mathbf{z}}   \widetilde{\mathbf{R}}[\mathbf{z}] \mathbf{w}\>\\&
 =-\Omega(\boldsymbol{\eta},D_{\mathbf{z}}^2\boldsymbol{\phi}[\mathbf{z}](\widetilde{\mathbf{z}},\mathbf{w}))+\< \boldsymbol{\eta}, D  _{\mathbf{z}}   \widetilde{\mathbf{R}}[\mathbf{z}] \mathbf{w}\>
\end{align*}
 and
 \begin{align*} &  \Omega( \dot{\boldsymbol{\eta}} ,D_{\mathbf{z}}\boldsymbol{\phi}[\mathbf{z}]\mathbf{w}) = -\Omega(  \boldsymbol{\eta}  ,D_{\mathbf{z}}^{2}\boldsymbol{\phi}[\mathbf{z}](\dot {\mathbf{z} }, \mathbf{w}))= -\Omega(  \boldsymbol{\eta}  ,D_{\mathbf{z}}^{2}\boldsymbol{\phi}[\mathbf{z}](\dot{\mathbf{z}}
-\widetilde{\mathbf{z}},\mathbf{w})) -\Omega(  \boldsymbol{\eta}  ,D_{\mathbf{z}}^{2}\boldsymbol{\phi}[\mathbf{z}]( \widetilde{\mathbf{z}},\mathbf{w}))  .
\end{align*}
    Thus
\begin{align}\label{eq:eqz}
\Omega(D_{\mathbf{z}}\boldsymbol{\phi}[\mathbf{z}](\dot{\mathbf{z}}-\widetilde{\mathbf{z}}),D_{\mathbf{z}}\boldsymbol{\phi}[\mathbf{z}]\mathbf{w})
=&\Omega(\boldsymbol{\eta},D^2_{\mathbf{z}}\boldsymbol{\phi}[\mathbf{z}](\dot{\mathbf{z}}
-\widetilde{\mathbf{z}},\mathbf{w}))
+\<\boldsymbol{\eta}, D  _{\mathbf{z}}   \widetilde{\mathbf{R}}[\mathbf{z}]\mathbf{w}\>
+\<\mathbf{F}[\mathbf{z},\boldsymbol{\eta}],D_{\mathbf{z}}\boldsymbol{\phi}[\mathbf{z}]\mathbf{w}\> .
\end{align}
Since  $\Omega(D_{\mathbf{z}}\boldsymbol{\phi}[\mathbf{z}]\cdot,D_{\mathbf{z}}\boldsymbol{\phi}[\mathbf{z}]\cdot)$ is a a symplectic form for  $\C^{N}$, taking $\|\mathbf{w}\|=1$ in an appropriate direction we obtain
\begin{align*}
\|\dot{\mathbf{z}}-\widetilde{\mathbf{z}}\|  \lesssim
  \delta\| \sech \(  \kappa x\) \boldsymbol{\eta}  \|_{L^2}+\| \sech \(  \kappa x\) \mathbf{F}[\mathbf{z},\boldsymbol{\eta}]\|_{L^2} .
\end{align*}
By    \eqref{eq:lem:estF1},  we have the conclusion.
\end{proof}
Lemma \ref{lem:modbound} completes the proof   of Proposition \ref{prop:modp}, recalling \eqref{eq:normk}.\qed

\section{  Technical lemmas I}\label{sec:tech1}

The following  is a slight refinement of a result in \cite{CM2109.08108}.
\begin{lemma}\label{lem:equiv_rho0}  Let $U\ge 0$ be a non--zero   potential $U\in L^1(\R , \R )$. Then
   there exists a constant $C  _{U}>0$ such that
for any    function $0\le W$  such that $ \<x\> W\in L^1(\R )$ then
 \begin{align}\label{eq:lem:rhoequiv0}
& \<  W f,f\> \le    C_U \( \| \<x\> W  \| _{L^1(\R )}  \|  f' \| _{L^2(\R )} ^2  +   \|   W  \| _{L^1(\R )}     \<   U  f,f\> \).\end{align}
In particular, we have
 \begin{align}\label{eq:lem:rhoequiv01}
& \|    \sech \( \frac{2}{A} x \)f\| _{L^2(\R )} ^2 \lesssim A^2    \|  f' \| _{L^2(\R )} ^2    +A     \|  \sech \( \kappa  x \) f  \| _{L^2(\R )} ^2 .\end{align}

\end{lemma}
\proof Let $J$ be a compact interval where $I  _{U}:= \int _{J}U(x) dx>0$.
  Let then  $x_0\in J$ s.t.
\begin{align*}
 |f(x_0)|^2\le I  _{U } ^{-1} \int _{J}|f(x )|^2U(x)  dx .
\end{align*}
Then,
\begin{align*}
  |f(x )| \le  |x-x_0|   ^{\frac{1}{2}}     \| f' \| _{L^2(\R )} + |f(x_0 )| \le  |x-x_0|   ^{\frac{1}{2}}     \| f' \| _{L^2(\R )} +I _U ^{-1/2} \<  U f,f\> ^{\frac{1}{2}}.
\end{align*}
Taking second power and multiplying by $W$
it is easy to conclude the following, which after integration yields  \eqref{eq:lem:rhoequiv0},
\begin{align*}
W(x)  |f(x )| ^2\le 2 \( 1   +|x_0| \) \<x\>W (x)    \| f' \| _{L^2(\R )} ^2    + 2   W (x)I _U ^{-1 } \<  U f,f\> .
\end{align*}
\qed

 We  will need  the following related technical result.
\begin{lemma}\label{lem:equiv_rho} There exists $A_0>0$ such that for any $A\ge A_0$,
 \begin{align}\label{eq:lem:rhoequiv}
& \|     \sech \(  \kappa x\)  f\|_{L^2}    \le      A  \(
\|  \sech \(   \frac{2}{A}x\)  f'\|_{L^2} +A^{-1}\| \sech   \(   \frac{2}{A}x\)  f\|_{L^2}  \) \text{ for any $f$.}
\end{align}
\end{lemma}
\proof
Taking $A_0=2/\kappa$, we have $\sech(\kappa x)\leq \sech(\frac{2}{A}x)$.
Thus, we have the conclusion by
\begin{align*}
 \|     \sech \(  \kappa x\)  f\|_{L^2}    \le      A  \cdot A^{-1}\| \sech   \(   \frac{2}{A}x\)  f\|_{L^2}  \leq A
  \(
 \|  \sech \(   \frac{2}{A}x\)  f'\|_{L^2} +A^{-1}\| \sech   \(   \frac{2}{A}x\)  f\|_{L^2}  \).
\end{align*}
Therefore, we have the conclusion.

\qed

\section{Proof of  Proposition \ref{prop:FGR}: the Fermi Golden Rule} \label{sec:FGR}   \label{sec:FGR}

To prove Proposition \ref{prop:FGR},    for the    $\mathbf{g}_{\mathbf{m}}$ in Assumption \ref{ass:FGR},    we consider
\begin{align}\label{eq:FGRfunctional}
\mathcal{J}_{\mathrm{FGR}}:=\Omega(\boldsymbol{\eta},\chi_A\sum_{\mathbf{m}\in \mathbf{R}_{\mathrm{min}}}\mathbf{z}^{\mathbf{m}}\mathbf{g}_{\mathbf{m}}).
\end{align}

Computing the time derivative of $\mathcal{J}_{\mathrm{FGR}}$, we have the following estimate.
\begin{lemma}\label{lem:FGR1}
We have
\begin{align}
\left|\dot{\mathcal{J}}_{\mathrm{FGR}}-\<\sum_{\mathbf{m}\in \mathbf{R}_{\mathrm{min}}}\mathbf{z}^{\mathbf{m}}\mathbf{G}_{\mathbf{m}},\sum_{\mathbf{m}\in \mathbf{R}_{\mathrm{min}}}\mathbf{z}^{\mathbf{m}}\mathbf{g}_{\mathbf{m}}\>\right|\lesssim
A^{-1/2}\( \sum_{\mathbf{m}\in \mathbf{R}_{\mathrm{min}}}|\mathbf{z}^{\mathbf{m}}|^2+\|\boldsymbol{\eta}\|_{\boldsymbol{ \Sigma }_A}^2\). \label{eq:lem:FGR11}
\end{align}
\end{lemma}

\begin{proof}
Differentiating $\mathcal{J}_{\mathrm{FGR}}$  and using \eqref{eq:modeq}, we have
\begin{align*}
\dot{\mathcal{J}}_{\mathrm{FGR}}=&
\Omega(\dot{\boldsymbol{\eta}},\chi_A\sum_{\mathbf{m}\in \mathbf{R}_{\mathrm{min}}}\mathbf{z}^{\mathbf{m}}\mathbf{g}_{\mathbf{m}})
+\Omega(\boldsymbol{\eta},\chi_A\sum_{\mathbf{m}\in \mathbf{R}_{\mathrm{min}}}D _{\mathbf{z}} \mathbf{z}^{\mathbf{m}}\widetilde{\mathbf{z}} \ \mathbf{g}_{\mathbf{m}})
\\&+\Omega(\boldsymbol{\eta},\chi_A\sum_{\mathbf{m}\in \mathbf{R}_{\mathrm{min}}}D _{\mathbf{z}} \mathbf{z}^{\mathbf{m}} \(\dot{\mathbf{z}}-\widetilde{\mathbf{z}}\)\mathbf{g}_{\mathbf{m}})=:A_1+A_2+A_3.\nonumber
\end{align*}
By Lemma \ref{lem:modbound} and Lemma \ref{lem:equiv_rho} and by \eqref{eq:relABg}, $A_3$ can be bounded by
\begin{align*}
|A_3|&\lesssim
 \|\boldsymbol{\eta}\chi_A\|_{L^1}\delta \|\dot{\mathbf{z}}-\widetilde{\mathbf{z}}\| _{\C^N}
\lesssim
 \delta ^2\| \sech \( \frac{2}{A}x\)\boldsymbol{\eta}\|_{L^2}  \|\sech \( \kappa x\) {\eta}_1\|_{L^2}
\\&\lesssim
  \delta ^2  A^2    \|    \boldsymbol{\eta}\|_{\boldsymbol{ \Sigma }_A}^2 \lesssim A^{-1/2}\|    \boldsymbol{\eta}\|_{\boldsymbol{ \Sigma }_A}^2.
\end{align*}
  By Equation  \eqref{eq:modeq}, we have
\begin{align*}
A_1=&\Omega(-D_{\mathbf{z}}\boldsymbol{\phi}[\mathbf{z}](\dot{\mathbf{z}}-\widetilde{\mathbf{z}}),\chi_A\sum_{\mathbf{m}\in \mathbf{R}_{\mathrm{min}}}\mathbf{z}^{\mathbf{m}}\mathbf{g}_{\mathbf{m}})
+
\<
	\mathbf{L}_1\boldsymbol{\eta},\chi_A\sum_{\mathbf{m}\in \mathbf{R}_{\mathrm{min}}}\mathbf{z}^{\mathbf{m}}\mathbf{g}_{\mathbf{m}}\>\\
	&
+\<  d\mathbf{f}[\boldsymbol{\phi}[\mathbf{z}]]\boldsymbol{\eta}+\mathbf{F}[z,\boldsymbol{\eta}]+\mathbf{R}[\mathbf{z}],\chi_A\sum_{\mathbf{m}\in \mathbf{R}_{\mathrm{min}}}\mathbf{z}^{\mathbf{m}}\mathbf{g}_{\mathbf{m}}\>
+\<\sum_{\mathbf{m}\in \mathbf{R}_{\mathrm{min}}}\mathbf{z}^{\mathbf{m}}\mathbf{G}_{\mathbf{m}},\chi_A\sum_{\mathbf{m}\in \mathbf{R}_{\mathrm{min}}}\mathbf{z}^{\mathbf{m}}\mathbf{g}_{\mathbf{m}}\>\\&
=A_{11}+A_{12}+A_{13}+A_{14}.
\end{align*}
By  Lemma    \ref{lem:modbound} and Lemma \ref{lem:equiv_rho} and by \eqref{eq:relABg}  we have
\begin{align*}
|A_{11}|&\lesssim
\|\dot{\mathbf{z}}-\widetilde{\mathbf{z}}\|   _{\C^N} \sum_{\mathbf{m}\in \mathbf{R}_{\mathrm{min}}}|\mathbf{z}^{\mathbf{m}}|\lesssim
\delta\(\|\sech (\kappa x)\boldsymbol{\eta}\|_{L^2}^2+\sum_{\mathbf{m}\in \mathbf{R}_{\mathrm{min}}}|\mathbf{z}^\mathbf{m}|^2\) \\&\le A^{-1/2}\( \sum_{\mathbf{m}\in \mathbf{R}_{\mathrm{min}}}|\mathbf{z}^{\mathbf{m}}|^2+\|\boldsymbol{\eta}\|_{\boldsymbol{ \Sigma }_A}^2\).
\end{align*}
By \eqref{eq:R1FRemainder} and Lemma  \ref{lem:estF}   we have
\begin{align*}
|A_{13}|&\lesssim  \sum_{\mathbf{m}\in \mathbf{R}_{\mathrm{min}}}|\mathbf{z}^\mathbf{m}| \(\| d\mathbf{f}[\boldsymbol{\phi}[\mathbf{z}]]\boldsymbol{\eta}\|_{L^1}+\| \mathbf{F}[z,\boldsymbol{\eta}]\chi_A\|_{L^1} +\|\mathbf{R}[z]\|_{L^1}\)\\&\lesssim A^{1/2}\delta\(\|\sech \( \frac{2}{A} x\)\boldsymbol{\eta}\|_{L^2}^2+\sum_{\mathbf{m}\in \mathbf{R}_{\mathrm{min}}}|\mathbf{z}^\mathbf{m}|^2\)\le A^{-1/2}\( \sum_{\mathbf{m}\in \mathbf{R}_{\mathrm{min}}}|\mathbf{z}^{\mathbf{m}}|^2+\|\boldsymbol{\eta}\|_{\boldsymbol{ \Sigma }_A}^2\).
\end{align*}
The term $A_{12}$ can be further decomposed as
\begin{align*}
A_{12}=\<\boldsymbol{\eta},\chi_A\sum_{\mathbf{m}\in \mathbf{R}_{\mathrm{min}}}\mathbf{z}^{\mathbf{m}}\mathbf{L}_1\mathbf{g}_{\mathbf{m}}\>+\<\boldsymbol{\eta},[\mathbf{L}_1,\chi_A]\sum_{\mathbf{m}\in \mathbf{R}_{\mathrm{min}}}\mathbf{z}^{\mathbf{m}}\mathbf{g}_{\mathbf{m}}\>=:A_{121}+A_{122}.
\end{align*}
By $[\mathbf{L}_1,\chi_A]=\begin{pmatrix}
-\chi_A''-2\chi_A'\partial_x & 0 \\ 0 & 0
\end{pmatrix}$,   we have the bound
\begin{align*}
|A_{122}|
&\lesssim \sum_{\mathbf{m}\in \mathbf{R}_{\mathrm{min}}}|\mathbf{z}^{\mathbf{m}}| \(\|\chi_A'' \eta_1\|_{L^1}+\|\chi_A' \eta_1'\|_{L^1}\)\\&\lesssim  \sum_{\mathbf{m}\in \mathbf{R}_{\mathrm{min}}}|\mathbf{z}^{\mathbf{m}}|(A^{-3/2}\|\sech \( \frac{2}{A} x\) \eta_1\|_{L^2}+A^{-1/2}\|\sech \( \frac{2}{A} x\) \eta_1'\|_{L^2})\\&
\lesssim A^{-1/2}\( \sum_{\mathbf{m}\in \mathbf{R}_{\mathrm{min}}}|\mathbf{z}^{\mathbf{m}}|^2+\|\sech \( \frac{2}{A} x\)\eta_1'\|_{L^2}^2+A^{-2}\|\sech \( \frac{2}{A} x\) \eta_1\|_{L^2}^2\) ,
\end{align*}
while we have, see Assumption \ref{ass:FGR},      \begin{align}\label{eq:a121}
  A_{121} = \<\boldsymbol{\eta},\chi_A\sum_{\mathbf{m}\in \mathbf{R}_{\mathrm{min}}}\mathbf{z}^{\mathbf{m}} \im (\mathbf{m}\cdot \boldsymbol{\lambda}) \mathbf{J} ^{-1}   \mathbf{g}_{\mathbf{m}}\>  .
\end{align}
The term
$A_{14}$ can be decomposed as
\begin{align}\nonumber
A_{14}&=\<\sum_{\mathbf{m}\in \mathbf{R}_{\mathrm{min}}}\mathbf{z}^{\mathbf{m}}\mathbf{G}_{\mathbf{m}},\sum_{\mathbf{m}\in \mathbf{R}_{\mathrm{min}}}\mathbf{z}^{\mathbf{m}}\mathbf{g}_{\mathbf{m}}\>-\<\sum_{\mathbf{m}\in \mathbf{R}_{\mathrm{min}}}\mathbf{z}^{\mathbf{m}}\mathbf{G}_{\mathbf{m}},(1-\chi_A)\sum_{\mathbf{m}\in \mathbf{R}_{\mathrm{min}}}\mathbf{z}^{\mathbf{m}}\mathbf{g}_{\mathbf{m}}\>
\\&=
\<\sum_{\mathbf{m}\in \mathbf{R}_{\mathrm{min}}}\mathbf{z}^{\mathbf{m}}\mathbf{G}_{\mathbf{m}},\sum_{\mathbf{m}\in \mathbf{R}_{\mathrm{min}}}\mathbf{z}^{\mathbf{m}}\mathbf{g}_{\mathbf{m}}\>
+A_{141},\label{KMFGR5}
\end{align}
where the 1st term of  line  \eqref{KMFGR5} is the main term appearing in \eqref{eq:lem:FGR11}.
Recalling $a_2=\frac{1}{2}\sqrt{m^2-\lambda_{ N}^2}$,
\begin{align*}
|A_{141}|\lesssim e^{-a_2A/2}\left|\sum_{\mathbf{m}\in \mathbf{R}_{\mathrm{min}}}\mathbf{z}^{\mathbf{m}}\right|^2 \lesssim A^{-1/2}\sum_{\mathbf{m}\in \mathbf{R}_{\mathrm{min}}}|\mathbf{z}^{\mathbf{m}}|^2.
\end{align*}
By the elementary identity $D_{\mathbf{z}}\mathbf{z}^{\mathbf{m}} \widetilde{\mathbf{z}}_0 =\im \mathbf{m}\cdot\boldsymbol{\lambda}\mathbf{z}^{\mathbf{m}}  $,
 the term $A_{2}$ can be decomposed as
\begin{align*}
A_{2}
=\< \mathbf{J}^{-1} \boldsymbol{\eta},\chi_A\sum_{\mathbf{m}\in \mathbf{R}_{\mathrm{min}}}\im \mathbf{m}\cdot\boldsymbol{\lambda}\mathbf{z}^{\mathbf{m}}\mathbf{g}_{\mathbf{m}}\>
+\Omega\(\boldsymbol{\eta},\chi_A\sum_{\mathbf{m}\in \mathbf{R}_{\mathrm{min}}}D_{\mathbf{z}}\mathbf{z}^{\mathbf{m}} \(\widetilde{\mathbf{z}}-\widetilde{\mathbf{z}}_0\)\mathbf{g}_{\mathbf{m}}\)=:A_{21}+A_{22},
\end{align*}
where
\begin{align*}
|A_{22}|
\lesssim
\delta \|\chi_A\boldsymbol{\eta}\|_{L^1} \sum_{\mathbf{m}\in \mathbf{R}_{\mathrm{min}}}|\mathbf{z}^{\mathbf{m}}|
\lesssim
A^{-1/2}\( \sum_{\mathbf{m}\in \mathbf{R}_{\mathrm{min}}}|\mathbf{z}^{\mathbf{m}}|^2+A^{-2}\|\sech \( \frac{2}{A} x\)\boldsymbol{\eta}\|_{L^2}^2\).
\end{align*}
Finally, by the antisymmetry  of $\mathbf{J}^{-1}(=-\mathbf{J})$ we have the cancellation  $A_{121}+A_{21}=0$.
Collecting all the estimates, we obtain  \eqref{eq:lem:FGR11}.
\end{proof}

We next take out the nonresonant terms from the main part of $\dot{\mathcal{J}}_{\mathrm{FGR}}$.

\begin{lemma}\label{lem:FGR2}
Let $\mathbf{m},\mathbf{n}\in \mathbf{R}_{\mathrm{min}}$ and $\mathbf{m}\neq \mathbf{n}$.
Then,
\begin{align*}
&
\mathbf{z}^{\mathbf{m}}\mathbf{z}^{\overline{\mathbf{n}}}=\frac{1}{\im\(\mathbf{m}\cdot\boldsymbol{\lambda}-\mathbf{n} \cdot\boldsymbol{\lambda}\)}\frac{d}{dt}\(\mathbf{z}^{\mathbf{m}}\mathbf{z}^{\overline{\mathbf{n}}}\)+r_{\mathbf{m},\mathbf{n}} \text{  where}\\& 
|r_{\mathbf{m},\mathbf{n}}|\lesssim \delta \sum_{\mathbf{m}\in \mathbf{R}_{\mathrm{min}}}|\mathbf{z}^{\mathbf{m}}|^2 + \delta \| \dot {\mathbf{z}} - \widetilde{\mathbf{z}}   \| ^2.
\end{align*}
\end{lemma}

\begin{proof}
We have
\begin{align*}
\frac{d}{dt}\(\mathbf{z}^{\mathbf{m}}\mathbf{z}^{\overline{\mathbf{n}}}\)=\im \( \mathbf{m} \cdot\boldsymbol{\lambda}- \mathbf{n} \cdot\boldsymbol{\lambda}\)\mathbf{z}^{\mathbf{m}}\mathbf{z}^{\overline{\mathbf{n}}}+D_{\mathbf{z}}\(\mathbf{z}^{\mathbf{m}}\mathbf{z}^{\overline{\mathbf{n}}}\)
\(\widetilde{\mathbf{z}}-\widetilde{\mathbf{z}}_0\)    +    D_{\mathbf{z}}\(\mathbf{z}^{\mathbf{m}}\mathbf{z}^{\overline{\mathbf{n}}}\) \(\dot {\mathbf{z}} - \widetilde{\mathbf{z}} \) .
\end{align*}
The estimate of $r_{\mathbf{m},\mathbf{n}}$ follows from Proposition \ref{prop:rp}.
\end{proof}

\begin{lemma}\label{lem:FGR3}
We have
\begin{align*}
 &
\left|\<
	\sum_{\mathbf{m}\in 	\mathbf{R}_{\mathrm{min}}}
		\mathbf{z}^{\mathbf{m}}\mathbf{G}_{\mathbf{m}},
	\sum_{\mathbf{m}\in \mathbf{R}_{\mathrm{min}}}
		\mathbf{z}^{\mathbf{m}}\mathbf{g}_{\mathbf{m}}
\>
-
\sum_{\mathbf{m}\in \mathbf{R}_{\mathrm{min}}}
	\gamma_{\mathbf{m}}|\mathbf{z}^{\mathbf{m}}|^2
-\frac{d}{dt}\Gamma
	\right|
\lesssim \delta \sum_{\mathbf{m}\in \mathbf{R}_{\mathrm{min}}}|\mathbf{z}^{\mathbf{m}}|^2 \text{  where}
\\&
\Gamma:=\sum_{\substack{\mathbf{m},\mathbf{n}\in \mathbf{R}_{\mathrm{min}}\\ \mathbf{m}\neq \mathbf{n}}}
	\<
		\frac{\mathbf{z}^{\mathbf{m}}\mathbf{z}^{\overline{\mathbf{n}}}}{\im\( \mathbf{m} \cdot\boldsymbol{\lambda}- \mathbf{n} \cdot\boldsymbol{\lambda}\)}\mathbf{G}_{\mathbf{m}},
	\mathbf{g}_{\mathbf{n}}
	\>.   
\end{align*}
\end{lemma}

\begin{proof}
It is immediate from Lemma \ref{lem:FGR2}.
\end{proof}

\begin{proof}[Proof of Proposition \ref{prop:FGR}]
The proof follows from  Lemmas \ref{lem:FGR1} and \ref{lem:FGR3} and the following estimates, due to \eqref{eq:orbbound},
\begin{align*}
	|\mathcal{J}_{\mathrm{FGR}}|&\lesssim \|\boldsymbol{\eta}\|_{L^2}\|\chi_A\|_{L^2}\sum_{\mathbf{m}\in 	\mathbf{R}_{\mathrm{min}}}|\mathbf{z}^{\mathbf{m}}|\lesssim \sqrt{A}\delta^3 \lesssim \delta^ 2 \text{   and}\\
	|\Gamma|&\lesssim \sum_{\mathbf{m}\in 	\mathbf{R}_{\mathrm{min}}}|\mathbf{z}^{\mathbf{m}}|^2\lesssim \delta^2.
\end{align*}
\end{proof}

\section{Proof of Proposition \ref{prop:1stvirial}.}
\label{sec:1virial}

We set, for the $\chi$ in \eqref{eq:chi},
\begin{align}\label{def:zetaphi}
\zeta_A(x):=\exp\(-\frac{|x|}{A}(1-\chi(x))\),\ \varphi_A(x):=\int_0^x \zeta_A^2(y)\,dy\ \text{and}\  S_A:=\frac{1}{2}\varphi_A'+\varphi_A\partial_x.
\end{align}
We will consider the   functionals
\begin{align*}
\mathcal{I}_{\mathrm{1st},1}:=\frac{1}{2}\Omega(\boldsymbol{\eta},S_A\boldsymbol{\eta}),\ \mathcal{I}_{\mathrm{1st},2}:=\frac{1}{2}\Omega\(\boldsymbol{\eta},\sigma_3\zeta_A^4 \boldsymbol{\eta}\),
\end{align*}
	where both $S_A$ and $\sigma_3\zeta_A^4$ are anti-symmetric w.r.t.\ $\Omega$.

\begin{lemma}\label{lem:1stV1}
We have
\begin{align} &
\|    \sech  \(   \frac{2}{A}x\)\eta_1'\|_{L^2}^2 + A^{-2}\|    \sech  \(   \frac{2}{A}x\) \eta_1\|_{L^2}^2\nonumber \\& \lesssim -\dot{\mathcal{I}}_{\mathrm{1st},1}+A^2\delta \|\boldsymbol{\eta}\| _{\boldsymbol{ \Sigma }_A}^2+\|   \sech  \(  \kappa x\) \boldsymbol{\eta}\|_{L^2}^2+\sum_{\mathbf{m}\in 	\mathbf{R}_{\mathrm{min}}}|\mathbf{z}^{\mathbf{m}}|^2. \label{eq:lem:1stV1}
\end{align}
\end{lemma}

\begin{proof}
We have
\begin{align}
\dot{\mathcal{I}}_{\mathrm{1st},1}&=
-\Omega(D\boldsymbol{\phi}[\mathbf{z}](\dot{\mathbf{z}}-\widetilde{\mathbf{z}}),S_A\boldsymbol{\eta})
+\<\mathbf{L}_1\boldsymbol{\eta} ,S_A\boldsymbol{\eta}\>
+\< \mathbf{f}[\boldsymbol{\phi}[\mathbf{z}]+\boldsymbol{\eta}]
-\mathbf{f}[\boldsymbol{\phi}[\mathbf{z}]] ,S_A\boldsymbol{\eta}\>
+\< \widetilde{\mathbf{R}}[\mathbf{z}],S_A\boldsymbol{\eta}\>
\nonumber\\&
=:B_1+B_2+B_3+B_4,\label{eq:1V1}
\end{align}
where $ \widetilde{\mathbf{R}}$ is defined in \eqref{eq:rtilde} and
\begin{align}\label{eq:ftilde} \widetilde{\mathbf{F}}[\mathbf{z},\boldsymbol{\eta}]:=\mathbf{f}[\boldsymbol{\phi}[\mathbf{z}]+\boldsymbol{\eta}]
-\mathbf{f}[\boldsymbol{\phi}[\mathbf{z}]]   .
\end{align}
The main term, $B_2$, can be decomposed as
\begin{align*}
B_2
&=\<L_1\eta_1,S_A\eta_1\>\\&=-\|(\zeta_A\eta_1)'\|_{L^2}^2-\frac{1}{2}\int \varphi_A V' \eta_1^2\,dx-\frac{1}{2}\int A^{-1}\(\chi'' |x|+2\chi' \frac{x}{|x|}\) \zeta_A\eta_1 ^{2} \,dx \\&=-\|(\zeta_A\eta_1)'\|_{L^2}^2+B_{21}+B_{22},
\end{align*}
where, $ |\varphi_A V'| \lesssim | x V'|\lesssim | x e ^{-a_1|x|}|$ and \eqref{eq:kappa} imply
\begin{align*}
|B_{21}|&\lesssim \|   \sech  \(  \kappa x\) \eta_1\|_{L^2}^2,
\end{align*}
and  by \eqref{eq:chi}
\begin{align*}
|B_{22}|&\lesssim A^{-1}\|   \sech  \(  \kappa x\)\eta_1\|_{L^2}^2.
\end{align*}
By Lemma \ref{lem:modbound}, we have
\begin{align*}
|B_{1}| \le \| \dot {\mathbf{z}} -\widetilde{\mathbf{z}}\|  \|   \boldsymbol{\eta}\| _{L^2 _{-\kappa}} \lesssim \delta \|   \boldsymbol{\eta}\| _{L^2 _{-\kappa}}^2.
\end{align*}
By    \eqref{eq:R1FRemainder} and \eqref{eq:rtilde}  we have
\begin{align*}
|B_4|&\lesssim  \|   \boldsymbol{\eta}\| _{L^2 _{-\kappa}}^2+\sum_{\mathbf{m}\in 	\mathbf{R}_{\mathrm{min}}}|\mathbf{z}^{\mathbf{m}}|^2.
\end{align*}
By $f(\phi   _1[\mathbf{z}] +\eta_1)-f(\phi   _1[\mathbf{z}] )=\int_0^1\int_0^1f''(s_1\phi[\mathbf{z}]_1+s_2\eta_1)\phi   _1[\mathbf{z}]\eta_1\,ds_1ds_2+f(\eta_1)$, we have
\begin{align*}
B_3=\<\int_0^1\int_0^1f''(s_1\phi _1[\mathbf{z}] +s_2\eta_1)\phi_1[\mathbf{z}] \eta_1\,ds_1ds_2,S_A\eta_1\>+ \<f(\eta_1),S_A\eta_1\>=B_{31}+B_{32}.
\end{align*}
By integration by parts,
\begin{align*}
B_{31}=-\frac{1}{2}\<\int_0^1\int^1\partial_x\(f''(s_1\phi   _1[\mathbf{z}] +s_2\eta_1)\phi    _1 [\mathbf{z}] \) \eta_1ds_1ds_2,\varphi_A\eta_1\>.
\end{align*}
Therefore, we have
\begin{align*}
|B_{31}| &\lesssim
\|  \cosh  \(  \kappa x\) \int_0^1\int^1\partial_x\(f''(s_1\phi   _1 [\mathbf{z}] +s_2\eta_1)\phi _1[\mathbf{z}] \)ds_1ds_2 \|_{L^\infty} \|  \sech  \(  \kappa x\)\eta_1^2\|_{L^1}\\&\lesssim
 \|\phi[\mathbf{z}] \|_{\boldsymbol{\Sigma }}\|   \sech  \(  \kappa x\)\eta_1\|_{L^2}^2\lesssim A^2 \delta\|\boldsymbol{\eta}\|_{\boldsymbol{ \Sigma }_A}^2,
\end{align*}
where the last inequality follows from Lemma \ref{lem:equiv_rho}.

\noindent For the pure in $\eta _1$  nonlinear term  $B_{32}$, by Lemma 2.7 of \cite{CM19SIMA}, which follows \cite{KMM3}, 
taking $A$ sufficiently large and $\delta_0$ sufficiently small, we have
\begin{align*}
|B_{32}|\leq o_{ \delta} (1)\|(\zeta_A\eta_1)'\|_{L^2}^2.
\end{align*}
Collecting  the estimates,
we have
\begin{align*}
&\|(\zeta_A\eta_1)' \|_{L^2}^2\lesssim   -\dot{\mathcal{I}}_{\mathrm{1st},1}+ \|\sech  \(  \kappa x\)\eta_1\|_{L^2}^2+A^2\delta \|\boldsymbol{\eta}\|_{A}^2
 +\|\sech  \(  \kappa x\)\boldsymbol{\eta}\|_{L^2}^2+\sum_{\mathbf{m}\in 	\mathbf{R}_{\mathrm{min}}}|\mathbf{z}^{\mathbf{m}}|^2.
\end{align*}
Finally, we claim the following, which is analogous to  (19) of \cite{KM22},
\begin{align}\label{eq:KM19}
\| \sech  \(   \frac{2}{A}x\) \eta_1'\|_{L^2}^2 + A^{-2}\| \sech  \(   \frac{2}{A}x\) \eta_1\|_{L^2}^2\lesssim \|(\zeta_A\eta_1)'\|_{L^2}^2 + A^{-1}\|\sech  \(  \kappa x\)\eta_1\|_{L^2}^2.
\end{align}
This yields \eqref{eq:lem:1stV1}. To prove \eqref{eq:KM19}, we set $w_1:= \zeta_A\eta_1$. We have
\begin{align*} & \int \zeta _A ^2 |w_1'|^2 dx = \int \zeta _A ^2       | \zeta _A \eta_1' + \zeta _A'\eta_1|^2   dx= \int \( \zeta _A ^4  \eta_1 ^{\prime 2} + \zeta _A ^3  \zeta _A'    (\eta_1^2)'  + \zeta _A ^2  \zeta _A ^{\prime 2} \eta_1^2\)      dx\\& = \int \( \zeta _A ^4  \eta_1 ^{\prime 2}
- \zeta _A ^3 \zeta _A '' \eta_1^2 -  2 \zeta _A ^2  \zeta _A ^{\prime 2} \eta_1^2\)      dx.
\end{align*}
This implies
\begin{align*}  \int   \zeta _A ^4  \eta_1 ^{\prime 2} \lesssim
\int \zeta _A ^2  w_1 ^{\prime 2} dx + A ^{-2}\int \zeta _A ^2 w_1^2 dx .
\end{align*}
Since by \eqref{eq:lem:rhoequiv01}  we have
\begin{align*}
&  A ^{-2}\int \zeta _A ^2 w_1^2 dx \lesssim      \|  w_1' \| _{L^2(\R )} ^2    +A ^{-1}    \|  \sech \( 2\kappa  x \) \zeta _A \eta_1   \| _{L^2(\R )} ^2\lesssim   \|  w_1' \| _{L^2(\R )} ^2    +A ^{-1} \|  \sech \(  \kappa  x \)   \eta_1   \| _{L^2(\R )} ^2  ,\end{align*}
we obtained the desired bound on the first term in the left hand side of \eqref{eq:KM19}.
We have
\begin{align*} A^{-2}\| \sech  \(   \frac{2}{A}x\) \eta_1\|_{L^2}^2\lesssim A ^{-2}\int \zeta _A ^2 w_1^2 dx\lesssim   \|  w_1' \| _{L^2(\R )} ^2    +A ^{-1} \|  \sech \(  \kappa  x \)   \eta_1   \| _{L^2(\R )} ^2\end{align*}
and hence we conclude the proof of \eqref{eq:KM19}.
\end{proof}

\begin{lemma}\label{lem:1stV2}
There exist $\delta_0>0$ and $A_0>0$ s.t.\ if $\delta<\delta_0$, for any $A>A_0$, we have
\begin{align}&
\|\sech  \(   \frac{2}{A}x\)\eta_2\|_{L^2}^2\nonumber \\& \lesssim -\dot{\mathcal{I}}_{\mathrm{1st},2}+\|\sech  \(   \frac{2}{A}x\) \eta_1'\|_{L^2}^2+\| \sech  \(   \frac{2}{A}x\)\eta_1\|_{L^2}^2+ \|\sech  \(  \kappa x\)\boldsymbol{\eta}\|_{L^2}^2+\sum_{\mathbf{m}\in 	\mathbf{R}_{\mathrm{min}}}|\mathbf{z}^{\mathbf{m}}|^2.\label{eq:lem:1stV2}
\end{align}

\end{lemma}
\begin{proof}
We have
\begin{align*}&
\dot{\mathcal{I}}_{\mathrm{1st},2} \\&=
-\Omega(D\boldsymbol{\phi}[\mathbf{z}](\dot{\mathbf{z}}-\widetilde{\mathbf{z}}),\sigma_3\zeta_A^4 \boldsymbol{\eta})
+\<\mathbf{L}_1\boldsymbol{\eta}
 ,
\sigma_3\zeta_A^4 \boldsymbol{\eta}\>+  \< \mathbf{f}[\boldsymbol{\phi}[\mathbf{z}]+\boldsymbol{\eta}]
-\mathbf{f}[\boldsymbol{\phi}[\mathbf{z}]] ,
\sigma_3\zeta_A^4 \boldsymbol{\eta}\>+ \< \widetilde{\mathbf{R}}[\mathbf{z}],
\sigma_3\zeta_A^4 \boldsymbol{\eta}\> \\&
=:C_1+C_2+C_3+C_4.
\end{align*}
For the main term $C_2$, we have
\begin{align*}
C_2=-\|\zeta_A^2\eta_2\|_{L^2}^2+\<L_1\eta_1,\zeta_A^4\eta_1\>
\end{align*}
and
\begin{align*}
|\<L_1\eta_1,\zeta_A^4\eta_1\>|\lesssim \|\sech  \(   \frac{2}{A}x\) \eta_1'\|_{L^2}^2+\|\sech  \(   \frac{2}{A}x\) \eta_1\|_{L^2}^2.
\end{align*}
For the remainder terms, we have
\begin{align*}
|C_1|&\lesssim \|\dot{\mathbf{z}}-\widetilde{\mathbf{z}}\|\|\sech  \(  \kappa x\)\boldsymbol{\eta}\|_{L^2}\lesssim \delta \|\sech  \(  \kappa x\)\boldsymbol{\eta}\|_{L^2}^2 ,\\
|C_3|&\lesssim  \delta\|\sech  \(   \frac{2}{A}x\) \eta_1\|_{L^2}^2,\\
|C_4 |&\lesssim
 \|     \sech  \(  \kappa x\)\boldsymbol{\eta}\|_{L^2}^2+\sum_{\mathbf{m}\in 	\mathbf{R}_{\mathrm{min}}}|\mathbf{z}^{\mathbf{m}}|^2
\end{align*}
Collecting the estimates, we have the conclusion.
\end{proof}

\begin{proof}[Proof of Proposition \ref{prop:1stvirial}]
From  $|\mathcal{I}_{\mathrm{1st},1}|\lesssim A\delta^2$, $|\mathcal{I}_{\mathrm{1st},2}|\lesssim \delta^2$, we have the conclusion from Lemmas \ref{lem:1stV1} and \ref{lem:1stV2}.
\end{proof}

\section{  Technical lemmas  II}\label{sec:tech}

We   consider
\begin{align} \label{def:Tg} &
\mathcal{T}:=\<\im \varepsilon \partial_x\>^{- {N} }\mathcal{A}^*.
\end{align}
The following lemma, where $P_c$ is the orthogonal projection on the continuous spectrum component of $L_1$, see \eqref   {def:orthecomp},
is proved in     \cite[Sect. 9]{CM2109.08108}.
\begin{lemma}\label{lem:coer6}
We have
\begin{align}\label{eq:Tinverse}
  \mathbf{u}=\prod_{j=1}^{N}R _{L_1}(\lambda_j^{2}) P_c \mathcal{A} \<   \im \varepsilon\partial_x\>^N  \mathcal{T} \mathbf{u}  \text{  for all }\mathbf{u}\in L^2_c(L_1).
\end{align}
\end{lemma}
\proof We provide the simple proof for completeness. We claim that we have
\begin{align}\label{eq:Tinverse1}
\mathcal{A} \mathcal{A}^*=A_1\circ \cdots \circ A_N \circ A_N^* \circ \cdots \circ A_1^* =\prod_{j=1}^{N}(L_1-\lambda _j^2).
\end{align}
To prove  \eqref{eq:Tinverse1},
we begin with the following, see the line  below \eqref{def:Ak},  \begin{align*}
A_N\circ A_N^* =L_N-\lambda _N^2.
\end{align*}
For $2\leq j \leq N$, we assume (notice that the Schr\"odinger operator $L_j$ is fixed)
\begin{align*}
A_j\circ \cdots \circ A_N \circ A_N^* \circ \cdots A_j^* = \prod_{k=j}^N(L_j-\lambda _k^2).
\end{align*}
Then, by
\begin{align*}
 A_{j-1}(L_j-\lambda _k^2)&=   A_{j-1}(A_{j-1}^*A_{j-1} +\lambda^{2}_{j-1}  -\lambda _k^2) =  (A_{j-1}A_{j-1}^*+\lambda^{2}_{j-1}-\lambda _k^2)       A_{j-1} \\& =(L_{j-1}-\lambda _k^2)A_{j-1},
\end{align*}
        we have
\begin{align*}&
A_{j-1}\circ \cdots \circ A_N \circ A_N^* \circ \cdots A_{j-1}^*  =A_{j-1} \prod_{k=j}^N(L_j-\lambda _k^2)   A_{j-1}^*  =\prod_{k=j}^N(L_{j-1}-\lambda _k^2) A_{j-1} \circ A_{j-1}^* \\& = \prod_{k=j}^N (L_{j-1}-\lambda _k^2)  \   (L_{j-1}-\lambda ^2 _{j-1})
=\prod_{k=j-1}^N(L_{j-1}-\lambda _k^2).
\end{align*}
Therefore, we have \eqref{eq:Tinverse1} by induction. Using it, from \eqref{def:Tg}  and $\mathbf{u}\in L^2_c(L_1)$ we have
\begin{align*}
\prod_{j=1}^{N}R _{L_1}(\lambda_j^{2}) P_c \mathcal{A} \<   \im \varepsilon\partial_x\>^N  \mathcal{T} \mathbf{u}   &=\prod_{j=1}^{N}R _{L_1}(\lambda_j^{2}) P_c A_1\circ \cdots \circ A_N \circ A_N^* \circ \cdots \circ A_1^*    \mathbf{u}\\&
=\prod_{j=1}^{N}R _{L_1}(\lambda_j^{2})P_c  \prod_{j=1}^{N}(L_1-\lambda _j^2) \mathbf{u}=P_c\mathbf{u} =\mathbf{u}.
\end{align*}

\qed

In \cite[Sect. 5]{CM2109.08108}  the following lemma  was proved.
\begin{lemma}\label{claim:l2boundIII}  Suppose that a Schwartz function   $\mathcal{V}\in \mathcal{S}(\R, \C)$   has the   property that for    $M\ge N+1$ its Fourier transform
satisfies
\begin{align}\label{eq2stestJ22III}& |      \widehat{{\mathcal{V}}} (k_1+ik_2) |\le  C_M \< k_1 \> ^{-M-1} \text{ for all $(k_1, k_2)\in \R \times [  \mathbf{b} ,  \mathbf{b} ]$    and}\\&  \widehat{{\mathcal{V}}}    \in C^0 ( \R \times [ -\mathbf{b} ,  \mathbf{b} ]) \cap H ( \R \times ( - \mathbf{b} ,  \mathbf{b} )), \nonumber
   \end{align}
  with $H(\Omega ) $ the set of holomorphic functions in an open subset $\Omega \subseteq \C$ and with a number  $\mathbf{b}>0$.
    Then, for multiplicative operators $  \cosh (\mathbf{b} x )$  and $   \cosh \(\frac{\mathbf{b}}{2} x\)$, we have
\begin{align}\label{eq2stestJ21II} &
 \|     \<  \im \varepsilon  \partial _x \> ^{-N}  [   \mathcal{V} , \< \im \varepsilon    \partial _x \> ^{N} ]     \cosh (\mathbf{b} x )    \| _{L ^{2  }(\R ) \to L ^{2 }(\R )} \le C_\mathbf{b} \varepsilon  ,  \\& \label{eq2stestJ21II121}
 \|   \cosh \(\frac{\mathbf{b}}{2} x\)     \<  \im \varepsilon  \partial _x \> ^{-N}  [   \mathcal{V} , \< \im \varepsilon    \partial _x \> ^{N} ]     \cosh \(\frac{\mathbf{b}}{2} x \)    \| _{L ^{2  }(\R ) \to L ^{2 }(\R )} \le C_\mathbf{b} \varepsilon  .\end{align}
  \end{lemma}
\proof For completeness we give the proof.
We start with  \eqref{eq2stestJ21II},    repeating the proof  from  \cite{CM2109.08108}. We have for $\sigma =0$
\begin{align*}&     \<  \im \varepsilon  \partial _x \> ^{-N}  [   \mathcal{V} , \< \im \varepsilon    \partial _x \> ^{N} ]   f  =   \int _{\R} dy  K ^{\sigma}(x,y)f(y),\end{align*}
where we set
\begin{align}& \label{kernelabs} K ^{\sigma}(x,y)= \int_{\R ^2}   e^{\im x k  - \im y \ell } \< \varepsilon  k \> ^{-\sigma} H(k,\ell )  dk d\ell \text{  with } \\& H(k,\ell )= \< \varepsilon k\> ^{-N} \widehat{\mathcal{V}}  (k-\ell )
 \( \< \varepsilon  k \> ^{ N}-\< \varepsilon  \ell \> ^{ N}\) .\nonumber
\end{align}
Notice that
\begin{align}&     H(k,\ell )= \varepsilon  H_1(k,\ell )  \text{  where }   H_1(k,\ell )= \< \varepsilon k\> ^{-N}     \widehat{\mathcal{V}}  (k-\ell )
 (k-\ell )   \frac{ P (\varepsilon  k,\varepsilon  \ell )} { \< \varepsilon k \> ^{ N}+\< \varepsilon  \ell \> ^{ N}},\label{functH}
\end{align}
 where $P $ is a $2N-1$ degree polynomial. Hence   the generalized integral in  \eqref{kernelabs} is absolutely convergent for $\sigma >0$. But also for $\sigma=0$
the operator
\begin{align*}&      T_{\sigma}f(x)=\int _{\R} dy f(y)  \int_{\R ^2}   e^{\im x k  - \im y \ell } \< \varepsilon  k \> ^{-\sigma} H_1(k,\ell )dkd\ell  \end{align*}
defines an operator $L^2(\R ) \to L^2(\R )$  of norm uniformly bounded in $\sigma \ge 0$. Let us focus now on $k=k_1+\im 0$ and $\ell =\ell_1-\im \mathbf{b}$
  \begin{align*}&      T_{\sigma} (\chi _{\R _+}f)(x)=\int _{\R _+} dy f(y)  e^{   - y  \mathbf{b} }     \int_{\R ^2}   e^{\im x k _1 - \im y \ell_1  } \< \varepsilon  k _1\> ^{-\sigma} H_1(k_1,\ell _1-\im \mathbf{b}  )dk_1d \ell _1 .\end{align*}
Now we claim that there exists $C>0$ such that
\begin{align}& \label{eq:clai1} \|  T_{\sigma}  \chi _{\R _+}f \| _{L^2\( \R \)} \le C  \|  e^{   -| x | \mathbf{b} }  f \| _{L^2\( \R _+\)} \text{  for all $\sigma >0$ and for all $f$.}
\end{align}
Set $g(y) =\chi _{\R _+} (y)f(y)  e^{   - y  \mathbf{b} }  $. Then
 \begin{align*}&      \widehat{T_{\sigma} (\chi _{\R _+}f)}(k_1)=      \int_{\R }   \< \varepsilon  k _1\> ^{-\sigma} H_1(k_1,\ell _1-\im \mathbf{b}  )  \widehat{g}(\ell _1 ) d \ell _1.
  \end{align*}
We claim that we have  \begin{align}&  \sup _{k_1\in \R} \int _{\R}   \< \varepsilon  k _1\> ^{-\sigma} |H_1(k_1,\ell _1-\im \mathbf{b}  )| d\ell _1 <C   ,\label{eq:young1}
\\&  \sup _{\ell_1\in \R} \int _{\R}   \< \varepsilon  k _1\> ^{-\sigma} |H_1(k_1,\ell _1-\im \mathbf{b}  )| dk _1 <C, \label{eq:young2}
\end{align}
for a fixed constant $C>0$.

 \noindent We have \begin{align*}
  \int _{\R} |H_1(k_1,\ell _1-\im \mathbf{b}  )| d\ell _1 &\lesssim   \int _{|\ell _1 |\in \left [ \frac{|k _1|}{2} , 2|k _1 | \right ]}     \< \varepsilon  k_1\> ^{-N}  \<k_1-\ell _1 \> ^{-M}
  \( \<  \varepsilon  k _1 \> ^{ N-1}+\left | \<  \varepsilon  \ell _1-\im \varepsilon \mathbf{b} \> \right |^{ N-1} \)      d\ell _1  \\& + \int _{|\ell _1 |\not \in \left [ \frac{|k _1 |}{2} , 2|k_1 | \right ]}      \< \varepsilon  k_1\> ^{-N}  \<k_1-\ell _1 \> ^{-M}
  \( \<  \varepsilon  k _1 \> ^{ N-1}+\left | \<  \varepsilon  \ell _1-\im \varepsilon \mathbf{b} \> \right |^{ N-1} \)      d\ell _1 .\end{align*}
The first integral can be bounded above by
\begin{align*}
   \int _{|\ell _1 |\in \left [ \frac{|k _1|}{2} , 2|k _1 | \right ]}     \< \varepsilon  k_1\> ^{-1}  \<k_1-\ell _1 \> ^{-M}
         d\ell _1   \le \| \< x \> ^{-M} \| _{L^1(\R )} ,\end{align*}
while the second can
be bounded above by
\begin{align*}
   \int _{\R }     \< \varepsilon  k_1\> ^{-N}   \frac{\<  \varepsilon  k _1 \> ^{ N-1}+\left | \<  \varepsilon  \ell _1-\im \varepsilon \mathbf{b} \> \right |^{ N-1}} { \<   k _1 \> ^{M}+\<    \ell _1\> ^{M}}
         d\ell _1   \le \| \< x \> ^{-M-1+N}\| _{L^1(\R )} .\end{align*}
So  \eqref{eq:young1}  is true for $C= \| \< x \> ^{-2} \| _{L^1(\R )}$.  Next we prove \eqref{eq:young2}.
We have \begin{align*}
  \int _{\R} |H_1(k_1,\ell _1-\im \mathbf{b}  )| dk _1 &\lesssim   \int _{|k _1 |\in \left [ \frac{|\ell  _1|}{2} , 2|\ell _1 | \right ]}     \< \varepsilon  k_1\> ^{-N}  \<k_1-\ell _1 \> ^{-M}
  \( \<  \varepsilon  k _1 \> ^{ N-1}+\left | \<  \varepsilon  \ell _1-\im \varepsilon \mathbf{b} \> \right |^{ N-1} \)      dk _1  \\& + \int _{|k _1 |\not \in \left [ \frac{|\ell _1 |}{2} , 2|\ell_1 | \right ]}      \< \varepsilon  k_1\> ^{-N}  \<k_1-\ell _1 \> ^{-M}
  \( \<  \varepsilon  k _1 \> ^{ N-1}+\left | \<  \varepsilon  \ell _1-\im \varepsilon \mathbf{b} \> \right |^{ N-1} \)      dk _1 .\end{align*}
The first integral can be bounded above by
\begin{align*}
   \int _{|k _1 |\in \left [ \frac{|\ell  _1|}{2} , 2|\ell _1 | \right ]}     \< \varepsilon  k_1\> ^{-1}  \<k_1-\ell _1 \> ^{-M}
         dk _1   \le \| \< x \> ^{-M} \| _{L^1(\R )} ,\end{align*}
while the second can
be bounded above by
\begin{align*}
   \int _{\R }     \< \varepsilon  k_1\> ^{-N}   \frac{\<  \varepsilon  k _1 \> ^{ N-1}+\left | \<  \varepsilon  \ell _1-\im \varepsilon \mathbf{b} \> \right |^{ N-1}} { \<   k _1 \> ^{M}+\<    \ell _1\> ^{M}}
         dk _1   \le \| \< x \> ^{-M-1+N}\| _{L^1(\R )} .\end{align*}
 So \eqref{eq:young2} is true for $C= \| \< x \> ^{-2} \| _{L^1(\R )}$. By Young's inequality, see Theorem 0.3.1 \cite{sogge}, we conclude that
 \eqref{eq:clai1}   is true  $C= \| \< x \> ^{-2} \| _{L^1(\R )}$. Proceeding similarly we can show
\begin{align*}&   \|  T_{\sigma}  \chi _{\R _-}f \| _{L^2\( \R \)} \le C  \|  e^{   - |x|  \mathbf{b} }  f \| _{L^2\( \R _-\)} \text{  for all $\sigma >0$ and for all $f$,}
\end{align*}
concluding, for  $C= \| \< x \> ^{-2} \| _{L^1(\R )}$,
\begin{align*}&   \|  T_{\sigma}  f \| _{L^2\( \R \)} \le C  \|  e^{   - |x|  \mathbf{b} }  f \| _{L^2\( \R  \)} \text{  for all $\sigma >0$ and for all $f$.}
\end{align*}
Now we show that this remains true for $\sigma =0$. For a sequence $\sigma _n\to 0^+$  then   $   T_{\sigma _n} f     \xrightarrow{n\to  +\infty }   T_{0} f$ point--wise  for  $f\in C^0_c(\R )$. Then  by the Fatou lemma and  by the density   of  $  C^0_c(\R )$  in $L^{2 }(\R )$
\begin{align}& \label{eq:clai3} \|  T_{0}  f \| _{L^2\( \R \)} \le C  \|  e^{   - |x|  \mathbf{b} }  f \| _{L^2\( \R  \)} \text{  for all   $f$.}
\end{align}
This is equivalent to \eqref{eq2stestJ21II}.

\noindent The proof of \eqref{eq2stestJ21II121} is similar, with the difference that for example
 \begin{align*}
 &     \chi _{\R _+} T_{\sigma} (\chi _{\R _+}f)(x)= e^{   - x  \frac{\mathbf{b}}{2} } \int _{\R _+} dy f(y)  e^{   - y  \frac{\mathbf{b}}{2} }     \int_{\R ^2}   e^{\im x k _1 - \im y \ell_1  } \< \varepsilon  k _1+ \im \varepsilon \frac{\mathbf{b}}{2}\> ^{-\sigma} H_1(k_1+ \im   \frac{\mathbf{b}}{2},\ell _1-\im \mathbf{b}  )dk_1d \ell _1,
 \end{align*}
and correspondingly we have there exists $C>0$ such that
\begin{align*}&   \| e^{     x  \frac{\mathbf{b}}{2} }\chi _{\R _+} T_{\sigma}  \chi _{\R _+}f \| _{L^2\( \R \)} \le C  \|  e^{   -| x | \frac{\mathbf{b}}{2} }  f \| _{L^2\( \R _+\)} \text{  for all $\sigma \ge 0$ and for all $f$,}
\end{align*}
which can be proved like \eqref{eq:clai1}, and so similarly the rest of the proof of \eqref{eq2stestJ21II121}.

\qed

We will need the following analogue  of Lemma \ref{claim:l2boundIII}.

\begin{lemma}\label{claim:l2boundIII1}  Suppose that a Schwartz function   $\mathcal{V}\in \mathcal{S}(\R, \C)$   has the   property that  ts Fourier transform
satisfies
\begin{align}\label{eq2stestJ22III1}& |      \widehat{{\mathcal{V}}} (k_1+ik_2) |\le  C_M \< k_1 \> ^{-2} \text{ for all $(k_1, k_2)\in \R \times [  \mathbf{b} ,  \mathbf{b} ]$    and}\\&  \widehat{{\mathcal{V}}}    \in C^0 ( \R \times [ -\mathbf{b} ,  \mathbf{b} ]) \cap H ( \R \times ( - \mathbf{b} ,  \mathbf{b} )), \nonumber
   \end{align}
   with a number  $\mathbf{b}>0$.
    Then
\begin{align}   \label{eq2stestJ21II1}
 \|         [   \mathcal{V} , \< \im \varepsilon    \partial _x \> ^{-N} ]    \cosh (\mathbf{b} y)    \| _{L ^{2  }(\R ) \to L ^{2 }(\R )} \le C_\mathbf{b}   .\end{align}
  \end{lemma}
\proof  The proof is similar to that of Lemma \ref{claim:l2boundIII}.    We have for $\sigma =0$
\begin{align*}&       [   \mathcal{V} , \<  \im \varepsilon  \partial _x \> ^{-N} ]   f  =   \int _{\R} dy  L ^{\sigma}(x,y)f(y),
\end{align*}
where we set
\begin{align}& \label{kernelabs1} L ^{\sigma}(x,y)= \int_{\R ^2}   e^{\im x k  - \im y \ell }  M_\sigma(k,\ell )  dk d\ell \text{  with } \\& M_\sigma(k,\ell )=   \widehat{\mathcal{V}}  (k-\ell )
 \( \< \varepsilon  k \> ^{ -N-\sigma}-\< \varepsilon  \ell \> ^{ -N-\sigma}\) .\nonumber
\end{align}
Hence   the generalized integral in  \eqref{kernelabs1} is absolutely convergent for $\sigma >0$. But also for $\sigma=0$
the operator
\begin{align*}
&      S_{\sigma}f(x)=\int _{\R} dy f(y)  \int_{\R ^2}   e^{\im x k  - \im y \ell }  M_0(k,\ell ), \end{align*}
defines an operator $L^2(\R ) \to L^2(\R )$, and the norm  is  uniformly bounded in $\sigma \ge 0$. Let us focus now on $k=k_1+\im 0$ and $\ell =\ell_1-\im \mathbf{b}$
  \begin{align*}&      S_{\sigma} (\chi _{\R _+}f)(x)=\int _{\R _+} dy f(y)  e^{   - y  \mathbf{b} }     \int_{\R ^2}   e^{\im x k _1 - \im y \ell_1  } M _{\sigma}(k_1,\ell _1-\im \mathbf{b}  )dk_1d \ell _1 .\end{align*}
Now we claim that there exists $C>0$ such that
\begin{align}& \label{eq:clai11} \|  S_{\sigma}  \chi _{\R _+}f \| _{L^2\( \R \)} \le C  \|  e^{   -| x | \mathbf{b} }  f \| _{L^2\( \R _+\)} \text{  for all $\sigma >0$ and for all $f$.}
\end{align}
Set like before  $g(y) =\chi _{\R _+} (y)f(y)  e^{   - y  \mathbf{b} }  $. Then
 \begin{align*}&      \widehat{S_{\sigma} (\chi _{\R _+}f)}(k_1)=      \int_{\R }    M _{\sigma}(k_1,\ell _1-\im \mathbf{b}  )  \widehat{g}(\ell _1 ) d \ell _1 .
 \end{align*}
We claim that  for a fixed constant $C>0$ we have  \begin{align}&  \sup _{k_1\in \R} \int _{\R}   | M _{\sigma}(k_1,\ell _1-\im \mathbf{b}  )| d\ell _1 <C   ,\label{eq:young11}
\\&  \sup _{\ell_1\in \R} \int _{\R}     | M _{\sigma}(k_1,\ell _1-\im \mathbf{b}  ))| dk _1 <C .\label{eq:young21}
\end{align}
 We have \begin{align*}
  \int _{\R} |M _{\sigma}(k_1,\ell _1-\im \mathbf{b}  ))| d\ell _1 &\lesssim    \int _{\R}     \<k_1-\ell _1 \> ^{-2}
  \( \<  \varepsilon  k _1 \> ^{- N -\sigma}+\left | \<  \varepsilon  \ell _1-\im \varepsilon \mathbf{b} \> ^{- N -\sigma} \right |\)      d\ell _1  \\& \lesssim    \int _{\R}     \<k_1-\ell _1 \> ^{-2}   d\ell _1    = \| \< x \> ^{-2} \| _{L^1(\R )} .\end{align*}
 So  \eqref{eq:young11}  is true for $C= \| \< x \> ^{-2} \| _{L^1(\R )}$.  Next we prove \eqref{eq:young21}.
Proceeding as above  \begin{align*}
  \int _{\R} |H_1(k_1,\ell _1-\im \mathbf{b}  )| dk _1 &\lesssim   \int _{\R }    \<k_1-\ell _1 \> ^{-2}
  \( \<  \varepsilon  k _1 \> ^{- N -\sigma}+\left | \<  \varepsilon  \ell _1-\im \varepsilon \mathbf{b} \> ^{- N -\sigma} \right |\)        dk _1  \\&  \lesssim    \int _{\R}     \<k_1-\ell _1 \> ^{-2}   dk _1    = \| \< x \> ^{-2} \| _{L^1(\R )}   .\end{align*}
So \eqref{eq:young11}--\eqref{eq:young21}  are true for  $C= \| \< x \> ^{-2} \| _{L^1(\R )}$ and by  Young's inequality   we conclude that
 \eqref{eq:clai11}   is true  $C= \| \< x \> ^{-2} \| _{L^1(\R )}$. Proceeding like above  we conclude
\begin{align*}&   \|  S_{\sigma}  f \| _{L^2\( \R \)} \le C  \|  e^{   - |x|  \mathbf{b} }  f \| _{L^2\( \R  \)} \text{  for all $\sigma >0$ and for all $f$ },
\end{align*}
which in turn, proceeding as above yields
\begin{align}& \label{eq:clai31} \|  S_{0}  f \| _{L^2\( \R \)} \le C  \|  e^{   - |x|  \mathbf{b} }  f \| _{L^2\( \R  \)} \text{  for all   $f$ },
\end{align}
and yields \eqref{eq2stestJ21II1}.

\qed

 We now apply Lemma \ref{claim:l2boundIII}  to obtain the following result.

\begin{lemma}\label{lem:coer6A}
We have
\begin{align}\label{eq:coer6A1} \|
   \prod_{j=1}^{N}R _{L_1}(\lambda_j^{2}) P_c \mathcal{A} \<   \im \varepsilon\partial_x\>^N   \mathbf{w}  \| _{L ^{2}_{-\kappa}} \lesssim   \|
     \mathbf{w}  \| _{L ^{2}_{-\frac{\kappa}{2}}} .
\end{align}
\end{lemma}
\proof We sketch the proof.  By a standard discussion  in    \cite[Appendix A]{CM2109.08108} which we skip here,   we have
   \begin{align*}& \prod _{j=1}^{N}R _{L_1}(\lambda_j^{2}) P_c = K_1...K_N, \end{align*}
  with integral operators with kernels satisfying  $|K_j(x,y) |\le C\< x-y\>  e^{-\sqrt{ m^2-\lambda _j^2} |x-y|} $     for  a fixed $C>0$. Then,
   by
    \begin{align*}&  \kappa \le \frac{m-\lambda _N}{10}  < \frac{\sqrt{ m^2-\lambda _j^2}}{10}, \end{align*}
     we have
 \begin{align*}&  \|  \sech  \( \kappa    x     \)  \prod _{j=1}^{N}R _{L_1}(\lambda_j^{2}) P_c \mathcal{A} \< \im \varepsilon  \partial _x \> ^{ N} v  \| _{L^2}\\& \lesssim  \|    \prod _{j=1}^{N}R _{L_1}(\lambda_j^{2}) P_c    \sech  \( \kappa    x     \)\mathcal{A} \< \im \varepsilon  \partial _x \> ^{ N} v  \| _{L^2}.\end{align*}
  We have
    \begin{align*}&            \sech  \( \kappa    x     \) \mathcal{A}  =P_{N}(x, \im  \partial _x )     \sech  \( \kappa    x     \) ,      \end{align*}
  for an $N$--th order differential operator with smooth and bounded coefficients.

\noindent Next, we write
 \begin{align*}&              \sech  \( \kappa    x     \)       \< \im \varepsilon  \partial _x \> ^{ N} =  \< \im \varepsilon  \partial _x \> ^{ N}  \sech  \( \kappa    x     \)
  + \< \im \varepsilon  \partial _x \> ^{ N} \< \im \varepsilon  \partial _x \> ^{ -N}      \left [   \sech  \( \kappa    x     \) ,     \< \im \varepsilon  \partial _x \> ^{ N} \right ] ,\end{align*}
so that
\begin{align*}
&    \left \|   \sech  \( \kappa    x     \)   \prod _{j=1}^{N}R _{L_1}(\lambda_j^{2}) P_c \mathcal{A} \< \im \varepsilon  \partial _x \> ^{ N} v\right  \| _{L^2(\R )}  \\&   \lesssim   \left \|     \prod _{j=1}^{N}R _{L_1}(\lambda_j^{2}) P_c  P_{N}(x, \im  \partial _x ) \< \im \varepsilon   \partial _x \> ^{ N} \sech  \( \kappa    x     \) v \right  \| _{L^2(\R )} \\&  +  \left \|     \prod _{j=1}^{N}R _{L_1}(\lambda_j^{2}) P_c  P_{N}(x, \im  \partial _x ) \< \im \varepsilon   \partial _x \> ^{ N} \< \im \varepsilon   \partial _x \> ^{ -N}      \left [   \sech  \( \kappa    x     \)  ,     \< \im \varepsilon   \partial _x \> ^{ N} \right ] v \right  \| _{L^2(\R )} \\& =:I+II
 .
\end{align*}
We have
\begin{align*}&  I \le  \left \|    \prod _{j=1}^{N}R _{L_1}(\lambda_j^{2}) P_c  P_{N}(x, \im  \partial _x ) \< \im \varepsilon   \partial _x \> ^{ N}  \right  \| _{L^2\to L^2} \left \|   \sech  \( \kappa    x     \) v \right  \| _{L^2(\R )} \le C  \left \|   \sech  \( \kappa    x     \)v \right  \| _{L^2(\R )},
 \end{align*}
with a fixed constant $C$ independent from $\varepsilon \in (0,1)$. Next, we have
\begin{align*}&  II \le   \left \|     \< \im \varepsilon   \partial _x \> ^{ -N}      \left [   \sech  \( \kappa    x     \)   ,     \< \im \varepsilon   \partial _x \> ^{ N} \right ] v \right  \| _{L^2(\R )} \le C \varepsilon  \left \|    \sech  \( 2^{-1}\kappa    x     \)v \right  \| _{L^2(\R )}, \end{align*}
by   Lemma \ref{claim:l2boundIII}, because  $\int e^{-\im kx}  \sech (x) dx = \pi \ \sech \( \frac{\pi}{2} k \)$, so that in the strip $k=k_1+ \im k_2$ with  $|k_2|\le \mathbf{b}:=\kappa /2 $,  then $\sech \( \frac{\pi}{2}    \ \frac{1 }{\kappa} k \) $ satisfies the estimates required on $\widehat{\mathcal{V}}$   in
\eqref{eq2stestJ22III}.
This completes the proof of \eqref{eq:coer6A1}.
\qed

As an application of \eqref{eq2stestJ21II1}, we prove the following.

\begin{lemma} \label{lem:KM1}
For any $u\in H^1$ we have
\begin{align}\label{eq:KM1}
\| \sech \(   \frac{4}{A}    x \)    \mathcal{T} u\|_{L^2}\lesssim
 \varepsilon^{-N}\|\sech \(    \frac{2}{A}    x\)u\|_{L^2},\\ \label{eq:KM2}
\| \sech \(    \frac{4}{A}    x  \) \partial_x\mathcal{T} u \|_{L^2}\lesssim
 \varepsilon^{-N}\| \sech \(  \frac{2}{A}    x\)u'\|_{L^2}+\|\sech \(  \kappa x\) u\|_{L^2}.
\end{align}

\end{lemma}
\proof  We have
\begin{align*}
&   \| \sech \(   \frac{4}{A}    x\)    \mathcal{T} u\|_{L^2} \le    \|     \< \im \varepsilon   \partial _x \> ^{ -N}    \sech \(  \frac{4}{A}    x\)    \mathcal{A}^* u\|_{L^2}
 + \|   \left [  \sech \(  \frac{4}{A}    x\) ,  \< \im \varepsilon   \partial _x \> ^{ -N} \right ]     \mathcal{A}^* u\|_{L^2} \\&=:I+II.
\end{align*}
We have
\begin{align*}
&       \sech \(  \frac{4}{A}    x\)    \mathcal{A}^* =P_N(\partial _x)  \sech \(   \frac{4}{A}    x\) ,      \end{align*}
  for an $N$--th order differential operator with smooth and bounded coefficients, uniformed bounded in $A\gg 1$, so that
  \begin{align*}
&   I \le \|      \< \im \varepsilon   \partial _x \> ^{ -N}   P_N(\partial _x)  \sech \(  \frac{4}{A}    x\)  u\|_{L^2} \lesssim \varepsilon ^{-N} \|       \sech \(   \frac{4}{A}    x\)  u\|_{L^2}.
\end{align*}
  We have
 \begin{align*}
&   II=
   \|   \left [  \sech \(  \frac{4}{A}    x \) ,  \< \im \varepsilon   \partial _x \> ^{ -N} \right ]     \mathcal{A}^* u\|_{L^2}  \\& \le  \|   \left [  \sech \(   \frac{4}{A}x   \) ,  \< \im \varepsilon   \partial _x \> ^{ -N} \right ] \cosh \(\frac{2}{A}x \) \| _{L^2\to L^2}   \|  \sech \(\frac{2}{A}x \)  \mathcal{A}^* u\|_{L^2}
   \lesssim  \|  \sech \(\frac{2}{A}x  \)  \mathcal{A}^* u\|_{L^2},
\end{align*}
   by   Lemma \ref{claim:l2boundIII}, because  $\int e^{-\im kx}  \sech (x) dx = \pi \ \sech \( \frac{\pi}{2} k \)$, so that in the strip $k=k_1+ \im k_2$ with  $|k_2|\le \mathbf{b}:=2/A $,  then $\sech \( \frac{\pi}{2}     \frac{A }{4} k \) $ satisfies the estimates required on $\widehat{\mathcal{V}}$   in
\eqref{eq2stestJ21II1}.
This completes the proof of \eqref{eq:KM1}. Now we turn to the proof of \eqref{eq:KM2}. We have
\begin{align*}
&  \mathcal{T} u=  \mathcal{T} \partial_x u +\< \im \varepsilon   \partial _x \> ^{ -N} [\partial_x, \mathcal{A}^*] u.
\end{align*}
By \eqref{eq:KM1}  we have
\begin{align*}
\| \sech \(   \frac{4}{A}    x \)    \mathcal{T}\partial_x u  \|_{L^2}\lesssim
 \varepsilon^{-N}\|\sech \(      \frac{2}{A}    x\)\partial_x u \|_{L^2} .
\end{align*}
We have \begin{align*}
&   [\partial _x,   \mathcal{A}^* ] =   \sum _{j=1}^{N} \prod _{i=0} ^{N-1-j} A^{*}_{N-i} \( \log \psi _j \) ^{\prime\prime}  \prod _{i= 1}^{j-1} A^{*}_{j-i} =P_N(\partial _x) \sech (\kappa x),
\end{align*}
with  the convention $ \prod _{i=0} ^{l}B_i =B_0\circ ...\circ B_l$, with $\psi _k$ the ground state of $L_k$
 and with $P_N(\partial _x) $ and $N$--th order differential operator with bounded coefficients. We then have
\begin{align*}
\| \sech \(   \frac{2}{A}    x \) \< \im \varepsilon   \partial _x \> ^{ -N} [\partial_x, \mathcal{A}^*] u  \|_{L^2}\le \|   \< \im \varepsilon   \partial _x \> ^{ -N} P_N(\partial _x) \sech (\kappa x) u  \|_{L^2} \lesssim
 \varepsilon^{-N}\| \sech (\kappa x)  u \|_{L^2} .
\end{align*}
\qed

As an application of  Lemma \ref{claim:l2boundIII1} we have the following.

\begin{lemma} \label{lem:KM2}
For any $u\in H^1$,
\begin{align}  \label{eq:KM3}&
\|     [\<\im \varepsilon \partial_x\>^{-N},V_D]\mathcal{A}^*u\|_{L^2}\lesssim
 \varepsilon \|\sech ( \kappa x) \mathcal{T}u\| _{L^2}, \\&  \label{eq:KM3b} \|  \cosh \( \frac{\kappa}{2} x\)   [\<\im \varepsilon \partial_x\>^{-N},V_D]\mathcal{A}^*u\|_{L^2}\lesssim
 \varepsilon \|\sech \( \frac{\kappa}{2} x\) \mathcal{T}u\| _{L^2}.
\end{align}
\end{lemma}
\proof  We have
\begin{align*}
\|      [\<\im \varepsilon \partial_x\>^{-N},V_D]\mathcal{A}^*u\|_{L^2} =\|     \<\im \varepsilon \partial_x\>^{-N}  [V_D, \<\im \varepsilon \partial_x\>^{ N} ]\mathcal{T} u\|_{L^2}   .
\end{align*}
Notice that
\begin{align*}
 V_D=V-2  \sum _{j=1}^{N}\( \log \psi _j\) ''.
\end{align*}
By  \eqref{eq:decay} and by the proof of Lemma 6 p.156 and Theorem 2 p. 167 \cite{DT79CPAM}   it then follows
\begin{align}\label{eq:decay1} |V ^{(l)}_D(x)|\le C  e^{-10 \kappa |x|} \text{ for all $0\le l\le N+1$.}
 \end{align}
This implies by an elementary integration by parts
\begin{align}\label{eq:decay2} | \widehat{V } _D(k_1+\im k_2)|\le C   \< k_1\> ^{-N-1} \text{ in the strip $|k_2|\le 9 \kappa$.}
 \end{align}
Then in particular, from \eqref{eq2stestJ21II} we obtain
\begin{align*} &
 \|     \<\im \varepsilon \partial_x\>^{-N}  [V_D, \<\im \varepsilon \partial_x\>^{ N} ] \cosh (\kappa x)   \sech (\kappa x) \mathcal{T} u\|_{L^2}   \lesssim \varepsilon \|    \sech (\kappa x) \mathcal{T} u\|_{L^2}   \text{   and similarly} \\&  \| \cosh \( \frac{\kappa}{2} x\)     \<\im \varepsilon \partial_x\>^{-N}  [V_D, \<\im \varepsilon \partial_x\>^{ N} ] \cosh \( \frac{\kappa}{2} x\)   \sech  \( \frac{\kappa}{2} x\)  \mathcal{T} u\|_{L^2}   \lesssim \varepsilon \|    \sech  \( \frac{\kappa}{2} x\) \mathcal{T} u\|_{L^2}.
\end{align*}
\qed

\section{Proof of Proposition \ref{prop:2ndvirial}} \label{proof:2dnv}

Using the operator $\mathcal{T}$  in  \eqref{def:Tg}, we consider the transformed variable
\begin{align}   \label{def:vBg}
\mathbf{v}:=\mathcal{T}\boldsymbol{\eta}.
\end{align}
Then,  for $\mathbf{L}_D:=\begin{pmatrix}
L_D & 0 \\ 0 & 1
\end{pmatrix} $ the variable $\mathbf{v}$ satisfies
\begin{align} \label{eq:vBg}
\dot{\mathbf{v}}=&
-\mathcal{T}D\boldsymbol{\phi}[\mathbf{z}](\dot{\mathbf{z}}-\widetilde{\mathbf{z}})
+\mathbf{J}\(\mathbf{L}_D\mathbf{v}
+\begin{pmatrix}
[\<\im \varepsilon\partial_x\>^{-N},V_D] & 0 \\ 0 & 0
\end{pmatrix}\mathcal{A}^*\boldsymbol{\eta}\)\\&
+\mathbf{J}       \mathcal{T} \(\mathbf{f}[\boldsymbol{\phi}[\mathbf{z}]+\boldsymbol{\eta} ]-\mathbf{f}[\boldsymbol{\phi}[\mathbf{z}]]
+\sum_{\mathbf{m}\in 	\mathbf{R}_{\mathrm{min}}}\mathbf{z}^{\mathbf{m}}\mathbf{G}_{\mathbf{m}}
+\mathbf{R}[\mathbf{z}]\)  . \nonumber
\end{align}

From Lemma \ref{lem:coer6A}, we have
\begin{align}
\|      \sech ( \kappa x)\boldsymbol{\eta}\|_{L^2}\lesssim \|\sech ( 2^{-1}\kappa x)\mathbf{v}\|_{L^2}.\label{eq:key1}
\end{align}
Set
\begin{align*}
\psi_{A,B}=\chi_A^2 \varphi_B,\  \widetilde{S}_{A,B}=\frac{1}{2}\psi_{A,B}'+\psi_{A,B}\partial_x,
\end{align*}
and consider the functionals
\begin{align*}
\mathcal{I}_{\mathrm{2nd},1}:=\frac{1}{2}\Omega(\mathbf{v},\widetilde{S}_{A,B}\mathbf{v}),\ \mathcal{I}_{\mathrm{2nd},2}:=\frac{1}{2}\Omega(\mathbf{v},\sigma_3 e^{-\kappa \<x\>}\mathbf{v}).
\end{align*}

\begin{lemma}\label{lem:2v1}
	We have
\begin{align} &\nonumber
\|\sech ( 2^{-1}\kappa x)v_1'\|_{L^2}^2+\|\sech ( 2^{-1}\kappa x)v_1\|_{L^2}^2+\dot{\mathcal{I}}_{\mathrm{2nd},1}\\&
\lesssim   \(\varepsilon^{-N}A^2\delta+A^{-1/2}\)\|\boldsymbol{\eta}\|  _{\boldsymbol{ \Sigma }_A}^2
+ \sum_{\mathbf{m}\in 	\mathbf{R}_{\mathrm{min}}}|\mathbf{z}^{\mathbf{m}}|^2.
\label{eq:lem:2v1} \end{align}
\end{lemma}

\begin{proof}
We have
\begin{align*}&
\dot{\mathcal{I}}_{\mathrm{2nd},1}= -\Omega( \mathcal{T}D\boldsymbol{\phi}[\mathbf{z}](\dot{\mathbf{z}}-\widetilde{\mathbf{z}}),\widetilde{S}_{A,B}\mathbf{v})+
\<\mathbf{L}_D\mathbf{v},\widetilde{S}_{A,B}\mathbf{v}\>\\& +\<\begin{pmatrix}
[\<\im \varepsilon\partial_x\>^{-N},V_D] & 0 \\ 0 & 0
\end{pmatrix}\mathcal{A}^*\boldsymbol{\eta},\widetilde{S}_{A,B}\mathbf{v}\>
  +\<\mathcal{T} \(  \mathbf{f}[\boldsymbol{\phi}[\mathbf{z}]+\boldsymbol{\eta}]
-\mathbf{f}[\boldsymbol{\phi}[\mathbf{z}]] \)
 ,\widetilde{S}_{A,B}\mathbf{v}\>  +\<\mathcal{T} \widetilde{\mathbf{R}}[\mathbf{z}] ,\widetilde{S}_{A,B}\mathbf{v}\>   \\&
=:D_1+D_2+D_3+D_4+D_5.
\end{align*}
 Following \cite{KM22},  for the main term $D_2$  we have
\begin{align*}
D_2=\<L_Dv_1,\widetilde{S}_{A,B}v_1\>=-\int \(  \xi _1 ^{\prime \prime2}+V_B \xi _1 \)\,dx+D_{21}   \text{  where } \xi _1=\chi_A\zeta_B v_1,
\end{align*}
and where
\begin{align*}&
 V_B = \frac{1}{2} \(      \frac{\zeta _B''}{\zeta _B}- \frac{(\zeta _B')^2}{\zeta _B^2}\) -\frac{1}{2}\  \frac{\varphi _B}{\zeta _B^2}V'_D \text{ and}\\&
D_{21}=\frac{1}{4}\int (\chi_A^2)'(\zeta_B^2)'v_1^2+\frac{1}{2}\int \(3(\chi_A')^2+\chi_A''\chi_A\)\zeta_B^2v_1^2-\int (\chi_A^2)'\varphi_B(v_1')^2+\frac{1}{4}\int (\chi_A^2)''' \varphi_B v_1^2.
\end{align*}
We claim
\begin{align}\label{eq:KMlemma3}
\int ( \xi _1 ^{\prime  2}+V_B \xi _1 )\,dx\gtrsim \(\| \sech \(  \frac{\kappa}{2} x\)v_1'\|_{L^2}^2+\|\sech \( \frac{\kappa}{2} x\)v_1\|^2\)-A^{-1}\|\boldsymbol{\eta}\|_{\boldsymbol{ \Sigma }_A}^2.
\end{align}
The proof is like in \cite[Lemma 3]{KM22}.  We have
\begin{align*} & \int _{|x|\le A}\sech \(  \kappa x\)v_1^2\le  \int _{|x|\le A}\sech \(  \frac{\kappa}{2} x\) \zeta ^2_B   v_1^2\le  \int _{|x|\le A}\sech \(  \frac{\kappa}{2} x\) \xi_1^2.
\end{align*}
We have
\begin{align*} & \int _{|x|\le A}\sech \(  \kappa x\)v_1 ^{\prime 2}\le  \int _{|x|\le A}\sech \(  \frac{\kappa}{2} x\) \( \xi _1' - \zeta '_Bv_1\) ^2
\lesssim  \int _{|x|\le A}\sech \(  \frac{\kappa}{2} x\)   (  \xi _1 ^{\prime 2} + \xi _1 ^{  2}    )    .
\end{align*}
We have
\begin{align*} & \int _{|x|\ge A}\sech \(  \kappa x\) \(  v_1 ^{\prime 2}+     v_1^2\) \le  \sech \(  \frac{\kappa}{2} A\)\int _{\R }\sech \( \frac{8}{A} x\) \(  v_1 ^{\prime 2}+     v_1^2\) dx \\&  \lesssim \sech \(  \frac{\kappa}{2} A\) \varepsilon ^{-N}\int _{\R }\sech \( \frac{4}{A} x\) \(  \eta_1 ^{\prime 2}+     \eta_1^2\) dx\le A ^{-1}\|\boldsymbol{\eta}\|_{\boldsymbol{ \Sigma }_A}^2.
\end{align*}
Finally, Lemma \ref{eq:lem:rhoequiv0} and Assumption \ref{ass:repulsive} imply
\begin{align*} & \int _{\R } \sech \(  \frac{\kappa}{2} x\)   (  \xi _1 ^{\prime 2} + \xi _1 ^{  2}    )  \lesssim  \int _{\R }( \xi _1 ^{\prime  2}+V_B \xi _1 )\,dx ,
\end{align*}
completing the proof of \eqref {eq:KMlemma3}.

We next claim the following, which is \cite[Lemma 4]{KM22},
\begin{align}\label{eq:KMlemma4}
|D_{21}|\lesssim A^{-1/2}\(\|\boldsymbol{\eta}\|    _{\boldsymbol{ \Sigma }_A}^2+\|\sech \( \kappa x\)\eta_1\|_{L^2}^2\) \lesssim A^{-1/2}\(\|\boldsymbol{\eta}\|    _{\boldsymbol{ \Sigma }_A}^2+   \varepsilon ^{-N}\|\sech \( \frac{\kappa}{2} x\)\eta_1\|_{L^2}^2\),
\end{align}
where the 2nd inequality follows from \eqref{eq:KM1}. Now we prove the first inequality.

\noindent Notice that $\chi _A(x)$  is constant for $|x|\not \in [A,2A]$, so that
\begin{align*} &  \text{$|(\chi_A^2)'(\zeta_B^2)'|\lesssim A ^{-1}B^{-1}e ^{-\frac{A}{B}}$,
$|\(3(\chi_A')^2+\chi_A''\chi_A\)\zeta_B^2|\lesssim A^{-2}e ^{-\frac{A}{B}}$}
\end{align*}
and since by $|\varphi  _B|\lesssim B$  we have $|(\chi_A^2)''' \varphi_B|\lesssim A^{-3}B$       and $|(\chi_A^2)'\varphi_B|\lesssim A^{-2}B$, we have
\begin{align*} &  \left | \frac{1}{4} (\chi_A^2)'(\zeta_B^2)'v_1^2+\frac{1}{2}  \(3(\chi_A')^2+\chi_A''\chi_A\)\zeta_B^2v_1^2-  (\chi_A^2)'\varphi_B(v_1')^2+\frac{1}{4}  (\chi_A^2)''' \varphi_B v_1^2 \right | \\& \lesssim   \frac{B}{A}  \sech \( \frac{8}{A}x\) \(    v_1 ^{\prime 2} + \frac{1}{A^2}   v_1 ^{  2}  \),
\end{align*}
by Lemma \ref{lem:KM1}    we have
\begin{align*} &  |D_{21}|\lesssim A^{-1/2} \(   \|   \sech \( \frac{4}{A}x\) v_1 '\| _{L^2}+A ^{-2}\|   \sech \( \frac{4}{A}x\) v_1  \| _{L^2}\)\\& \lesssim  A^{-1/2}\varepsilon^{-N} \(  \| \sech \(  \frac{2}{A}    x\)\eta _1'\|_{L^2}^2 +  A^{-2}\|\sech \(    \frac{2}{A}    x\)\eta _1\|_{L^2}^2 + \|\sech \(  \kappa x\) \eta _1\|_{L^2}^2 \),
\end{align*}
which yields the desired inequality \eqref{eq:KMlemma4}.

\noindent By Lemma \ref{lem:modbound} and by an analogue to \eqref{eq:KM1}, we have
\begin{align*}
|D_1|\lesssim |\dot{\mathbf{z}}-\widetilde{\mathbf{z}}| \|  \sech \( 2\kappa x\)    \mathbf{v}\|_{L^2}\lesssim \delta  \|\sech \(  \kappa x\) {\eta}_1\|_{L^2} \|\sech \( 2\kappa x\) \mathbf{v}\|_{L^2} \lesssim \delta \varepsilon ^{-N} \|\sech \(  \kappa x\) {\eta}_1\|_{L^2}^2 .
\end{align*}
By Lemma \ref{lem:KM2},  we have
\begin{align*}&
|D_3|  =|\< [\<\im \varepsilon \partial_x\>^{-N},V_D] \mathcal{A}^*\eta _1, \widetilde{S} _{A,B}v_1\>| \\& \le \| \cosh \( \frac{\kappa}{2} x\)[\<\im \varepsilon \partial_x\>^{-N},V_D] \mathcal{A}^*\eta _1 \| _{L^2} \| \sech \( \frac{\kappa}{2} x\) \widetilde{S} _{A,B}v_1\| _{L^2}
\\&
\leq   \varepsilon   \| \sech \( \frac{\kappa}{2} x\)v_1\|_{L^2}  \(\|\sech \( \frac{\kappa}{2} x\) v_1'\|_{L^2}+\|\sech \( \frac{\kappa}{2} x\)v_1\|_{L^2}\)\\&
\lesssim\varepsilon \(\|    \sech \( \frac{\kappa}{2} x\)v_1'\|_{L^2}^2+\|\sech \( \frac{\kappa}{2} x\) v_1\|_{L^2}^2\) ,
\end{align*}
where the upper bound can be absorbed inside the left hand side of \eqref{eq:lem:2v1}.

 \noindent Like in Lemma \ref{lem:1stV1}, we have
 \begin{align*}& D_4 =\<\int_0^1\int_0^1f''(s_1\phi  _1[\mathbf{z}] +s_2\eta_1)\phi   _1 [\mathbf{z}] \eta_1\,ds_1ds_2, \widetilde{S} _{A,B}v_1\>
 + \<f(\eta_1), \widetilde{S} _{A,B}v_1\>=:D_{41}+D_{42}.
\end{align*}
Ignoring the irrelevant $ds_1ds_2$ integral, we have
\begin{align*}& |D_{41}|  \lesssim \|  \cosh\(  2\kappa  x \)\(f''(s_1\phi   _1 [\mathbf{z}] +s_2\eta_1)\phi_1[\mathbf{z}] \sech \(  \kappa  x \)\eta_1\)\|_{L^2}
\|   \sech \(  \kappa  x \)\widetilde{S}_{A,B}v_1\|_{L^2}\\& \lesssim  \| \mathbf{z} \|       \| \sech \(  \kappa  x \)\eta_1 \| _{L^2}
 \( \|   \sech \(  \kappa  x \) v_1' \|_{L^2}   +  \|   \sech \(  \kappa  x \) v_1  \|_{L^2} \) \\&
  \lesssim \delta \varepsilon^{-N}\( \|   \sech \(  \kappa  x \) v_1' \|_{L^2} ^2  +  \|   \sech \(  \kappa  x \) v_1  \|_{L^2}^2 \) ,
\end{align*}
which can be absorbed inside the left hand side of \eqref{eq:lem:2v1}. Next, we have
\begin{align*}& |D_{42}|   =
 |\< \sech \(  \frac{2}{A} x \) f(\eta_1),  \cosh\(  \frac{2}{A} x \) \(\frac{1}{2} \( \chi_A^2 \varphi_B    \) '
 + \chi_A^2 \varphi_B\partial_x \)v_1\> |\\&  \lesssim \| \eta _1\| _{L^\infty} \| \sech \(  \frac{2}{A} x \)\eta _1\| _{L^2}
 \times \\& \( \|\cosh\(  \frac{6}{A} x \)\psi_{A,B}'\|_{L^\infty}\|\sech \(  \frac{4}{A} x \) v_1\|_{L^2}
+\|\cosh\(  \frac{6}{A} x \)\psi_{A,B}\|_{L^\infty}\|\sech \(  \frac{4}{A} x \)v_1'\|_{L^2}\)
\\& \lesssim A \delta\| \boldsymbol{\eta} \|_{ \boldsymbol{ \Sigma }_A} \(  \|\sech \(  \frac{4}{A} x \) v_1\|_{L^2}
+ \|\sech \(  \frac{4}{A} x \)v_1'\|_{L^2}\)   \\& \lesssim  \varepsilon^{-N} A \delta\| \boldsymbol{\eta} \|_{ \boldsymbol{ \Sigma }_A} \(
\|\sech \(  \frac{4}{A} x \) \eta_1\|_{L^2}
+ \|\sech \(  \frac{4}{A} x \)\eta_1'\|_{L^2}\) \lesssim  \varepsilon^{-N} A^2 \delta \| \boldsymbol{\eta} \|_{ \boldsymbol{ \Sigma }_A}^2.
\end{align*}
Finally,  we consider
\begin{align*}
 D_5 =
 \< \sum_{\mathbf{m}\in 	\mathbf{R}_{\mathrm{min}}}\mathbf{z}^{\mathbf{m}}\mathcal{T}\mathbf{G}_{\mathbf{m}} ,\widetilde{S}_{A,B}\mathbf{v}\>
   +\<\mathcal{T}  {\mathbf{R}}[\mathbf{z}] ,\widetilde{S}_{A,B}\mathbf{v}\> =:D_{51}+D_{52}.
\end{align*}
We focus on $D_{51}$ which is the main term. We have
\begin{align*}
  |
 \<  \mathbf{z}^{\mathbf{m}}\mathcal{T}\mathbf{G}_{\mathbf{m}} ,\widetilde{S}_{A,B}\mathbf{v}\> |\le | \mathbf{z}^{\mathbf{m}}| \| \cosh \( \kappa x\)
 \widetilde{S}_{A,B}\mathcal{T}\mathbf{G}_{\mathbf{m}} \| _{L^2} \| \sech \( \kappa x\)  \mathbf{v} \| _{L^2}
 \lesssim \frac{1}{\mu} | \mathbf{z}^{\mathbf{m}}| ^2 + \mu  \| \sech \( \kappa x\)  \mathbf{v} \| _{L^2}^2,
\end{align*}
where for $\mu $ small enough the last term can be absorbed in the left hand side of \eqref{eq:lem:2v1}.

Collecting the estimates, we have the conclusion.
\end{proof}

\begin{lemma}\label{lem:2v2}
	We have
	\begin{align}\label{eq:lem:2v2}
		\|e^{-\kappa\<x\>/2}v_2\|_{L^2} + \dot{\mathcal{I}}_{\mathrm{2nd},2}\lesssim& \|e^{-\kappa \<x\>/2}v_1'\|_{L^2}^2+
\|e^{-\kappa \<x\>/2}v_1\|_{L^2}^2+\sum_{\mathbf{m}\in 	\mathbf{R}_{\mathrm{min}}}|\mathbf{z}^{\mathbf{m}}|^2+
   \delta A
   \| \boldsymbol{\eta} \|_{ \boldsymbol{ \Sigma }_A}^2      .
	\end{align}
\end{lemma}

\begin{proof}
	Differentiating $\mathcal{I}_{\mathrm{2nd},2}$, we have
	\begin{align*}
		\dot{\mathcal{I}}_{\mathrm{2nd},2}=&
-\Omega(\mathcal{T} D\boldsymbol{\phi}[\mathbf{z}]
(\dot{\mathbf{z}}-\widetilde{\mathbf{z}}),\sigma_3e^{-\kappa\<x\>}\mathbf{v})+
\<\mathbf{L}_D\mathbf{v},\sigma_3e^{-\kappa\<x\>}\mathbf{v}\>\\&+
\<[\<\im\varepsilon\partial_x\>^{-N},V_D]\mathcal{A}^*\eta_1,\sigma_3e^{-\kappa\<x\>}v_1\>+
\<\mathcal{T} \(  f[\boldsymbol{\phi}[\mathbf{z}]+\boldsymbol{\eta}]
-f[\boldsymbol{\phi}[\mathbf{z}]] \),  e^{-\kappa\<x\>}v_1\> \\& +\<\mathcal{T}
\widetilde{\mathbf{R}}[\mathbf{z}] ,\sigma_3 e^{-\kappa\<x\>}\mathbf{v}\>  =:E_1+E_2+E_3+E_4+E_5.
	\end{align*}
The main term is
\begin{align*}
	E_2=-\|e^{-\kappa\<x\>/2}v_2\|_{L^2}^2+\<L_D v_1, e^{-\kappa\<x\>}v_1\>=-\|e^{-\kappa\<x\>/2}v_2\|_{L^2}^2+E_{21},
\end{align*}
with
\begin{align*}
	|E_{21}|\lesssim  \|e^{-\kappa \<x\>/2}v_1'\|_{L^2}^2+\|e^{-\kappa \<x\>/2}v_1\|_{L^2}^2.
\end{align*}
By Lemma \ref{lem:modbound}, we have
\begin{align*}
	|E_1|\lesssim \delta \|e^{-\kappa\<x\>/2}\mathbf{v}\|_{L^2}  \|e^{-\kappa\<x\>}\boldsymbol{\eta}\|_{L^2}
 \lesssim \delta \varepsilon ^{-N}\|e^{-\kappa\<x\>/2}\mathbf{v}\|_{L^2}^2.
\end{align*}
By \eqref{eq:KM3b}, we have
\begin{align*}
	&|E_3|  = |\<  [\<\im\varepsilon\partial_x\>^{-N},V_D]\mathcal{A}^*\eta_1,\sigma_3e^{-\kappa\<x\>}v_1\>| \lesssim
 \varepsilon \|e^{-\frac{\kappa}{2}\<x\> }v_1\|_{L^2} \|e^{-\kappa\<x\> }v_1\|_{L^2}    \le  \varepsilon \|e^{-\frac{\kappa}{2}\<x\> }v_1\|_{L^2}^2.
\end{align*}
We write
\begin{align*}
	& E_4=\<\int_0^1\int_0^1f''(s_1\phi  _1[\mathbf{z}]+s_2\eta_1)\phi_1 [\mathbf{z}] \eta_1\,ds_1ds_2,  e^{-\kappa\<x\>}v_1\>
 + \<f(\eta_1), \sigma_3e^{-\kappa\<x\>}v_1\>=:E_{41}+E_{42}.
\end{align*}
Ignoring the irrelevant $ds_1ds_2$ integral, we have
\begin{align*}& |E_{41}|  \lesssim \|  \(f''(s_1\phi   _1[\mathbf{z}] +s_2\eta_1)\cosh \(  \kappa  x \)\phi  _1[\mathbf{z}]
\sech \(  \kappa  x \)\eta_1\)\|_{L^2}
\|  e^{-\kappa\<x\>}v_1\|_{L^2}\\& \lesssim  \| \mathbf{z} \|       \| \sech \(  \kappa  x \)\eta_1 \| _{L^2}
   \|  e^{-\kappa\<x\>}v_1 \|_{L^2} \lesssim \delta  \| \sech \(  \frac{\kappa}{2}  x \)v_1 \| _{L^2}^2.
\end{align*}
We have
\begin{align*}& |E_{42}|   =
 |\<   f(\eta_1),  e^{-\kappa\<x\>}v_1\> |\\&  \lesssim \| \eta _1\| _{L^\infty} \| \sech \(  \frac{2}{A} x \)\eta _1\| _{L^2}
  \| \sech \(  \frac{\kappa}{2} x \)v _1\| _{L^2}  \lesssim \delta A \( \| \sech \(  \frac{\kappa}{2} x \)v _1\| _{L^2}  ^2 +
   \| \boldsymbol{\eta} \|_{ \boldsymbol{ \Sigma }_A}^2\) .
\end{align*}
 We have \begin{align*}
 E_5 =
 \< \sum_{\mathbf{m}\in 	\mathbf{R}_{\mathrm{min}}}\mathbf{z}^{\mathbf{m}}\mathcal{T}\mathbf{G}_{\mathbf{m}} ,\sigma_3e^{-\kappa\<x\>}\mathbf{v}\>
   +\<\mathcal{T}  {\mathbf{R}}[\mathbf{z}] ,\sigma_3e^{-\kappa\<x\>}\mathbf{v}\> =:E_{51}+E_{52}.
\end{align*}
We focus on $D_{51}$ which is the main term, the other being simpler.
 We have
\begin{align*}&
  |
 \<  \mathbf{z}^{\mathbf{m}}\mathcal{T}\mathbf{G}_{\mathbf{m}} ,,\sigma_3e^{-\kappa\<x\>}\mathbf{v}\> |
 \le | \mathbf{z}^{\mathbf{m}}| \|
  \mathcal{T}\mathbf{G}_{\mathbf{m}} \| _{L^2} \| \sech \( \kappa x\)  \mathbf{v} \| _{L^2}
 \lesssim \frac{1}{\mu} | \mathbf{z}^{\mathbf{m}}| ^2 + \mu  \| \sech \( \kappa x\)  \mathbf{v} \| _{L^2}^2 \\&
 = \frac{1}{\mu} | \mathbf{z}^{\mathbf{m}}| ^2 + \mu  \| \sech \( \kappa x\)   {v}_1 \| _{L^2}^2+ \mu  \| \sech \( \kappa x\)   {v}_2 \| _{L^2}^2,
\end{align*}
where for $\mu $ small enough the very  last term  in $v_2$ can be absorbed in the left   hand side of  \eqref{eq:lem:2v2}
Collecting the estimates, we have the conclusion.
\end{proof}

Combining Lemmas \ref{lem:2v1} and \ref{lem:2v2}, we have
\begin{lemma}\label{lem:2v3}
	For any $\mu>0$, we have
	\begin{align*}
		&\int_0^T\(\| \sech \( \frac{\kappa}{2}x \)v_1'\|_{L^2}^2+\|\sech \( \frac{\kappa}{2}x \)\mathbf{v}\|_{L^2}^2\)
\lesssim B\varepsilon^{-N}\delta^2 \\&\quad+\(\varepsilon^{-1}A^2\delta+A^{-1/2}\)\int_0^T\|\boldsymbol{\eta}\|_A^2+ \sum_{\mathbf{m}\in 	\mathbf{R}_{\mathrm{min}}}\|\mathbf{z}^{\mathbf{m}}\|_{L^2(0,T)}^2.
	\end{align*}
\end{lemma}
\begin{proof}
	The claim follows from Lemmas  \ref{lem:2v1} and \ref{lem:2v2} and
	\begin{align}
		|\mathcal{I}_{\mathrm{2nd},1}|\lesssim B\varepsilon^{-N}\delta^2,\
		|\mathcal{I}_{\mathrm{2nd},2}|\lesssim \varepsilon^{-N}\delta^2.
	\end{align}
\end{proof}

\begin{proof}[Proof of Proposition \ref{prop:2ndvirial}]
	It  is a consequence of Lemma \ref{lem:2v3}     and inequality \eqref{eq:key1}.
\end{proof}

\section*{Acknowledgments}
C. was supported   by the Prin 2020 project \textit{Hamiltonian and Dispersive PDEs} N. 2020XB3EFL.
M. was supported by the JSPS KAKENHI Grant Number 19K03579 and G19KK0066A.

Department of Mathematics and Geosciences,  University
of Trieste, via Valerio  12/1  Trieste, 34127  Italy.
{\it E-mail Address}: {\tt scuccagna@units.it}

Department of Mathematics and Informatics,
Graduate School of Science,
Chiba University,
Chiba 263-8522, Japan.
{\it E-mail Address}: {\tt maeda@math.s.chiba-u.ac.jp}

Department of Mathematics and Geosciences,  University
of Trieste, via Valerio  12/1  Trieste, 34127  Italy.
{\it E-mail Address}: {\tt STEFANO.SCROBOGNA@units.it}

\end{document}